\documentclass[reqno]{amsart}

\usepackage{amsmath,amssymb,amscd,amsthm,amsxtra,latexsym,dsfont}

\newcommand{\R}{\mathbb{R}}\newcommand{\T}{\mathbb{T}}
\newcommand{\D}{\mathbb{D}}
\newcommand{\Z}{\mathbb{Z}}
\newcommand{\ang}[1]{ \left< {#1} \right>}  
\newcommand{\supp}{{\mathrm{supp}}}
\newcommand{\dd}{d}
\newcommand{\ee}{e}
\newcommand{\ii}{i}
\newcommand{\deq}{:=}
\usepackage{graphicx, color}

\definecolor{darkred}{rgb}{0.9,0,0.3}
\definecolor{darkblue}{rgb}{0,0.3,0.9}
\definecolor{darkgreen}{rgb}{0,0.8,0.2}

\definecolor{vdarkred}{rgb}{0.6,0,0.2}
\definecolor{vdarkblue}{rgb}{0,0.2,0.6}
\usepackage[pdftex, colorlinks, linkcolor=vdarkblue,citecolor=vdarkred]{hyperref}
\DeclareMathOperator{\tr}{Tr}

\makeatletter
\DeclareRobustCommand\widecheck[1]{{\mathpalette\@widecheck{#1}}}
\def\@widecheck#1#2{%
  \box\z@\hbox{\m@th$#1#2$}%
  \box\tw@\hbox{\m@th$#1%

    \widehat{%
      \vrule\@width\z@\@height\ht\z@
      \vrule\@height\z@\@width\wd\z@}$}%
  \dp\tw@-\ht\z@ \@tempdima\ht\z@ \advance\@tempdima2\ht\tw@
  \divide\@tempdima\thr@@ \box\tw@\hbox{%

    \raise\@tempdima\hbox{\scalebox{1}[-1]{\lower\@tempdima\box\tw@}}}%
  {\ooalign{\box\tw@ \cr \box\z@}}} \makeatother

\newtheorem{theorem}{Theorem}[section]
 \newtheorem{lemma}[theorem]{Lemma}
\newtheorem{proposition}[theorem]{Proposition}
\newtheorem{remark}[theorem]{Remark} 

\newtheorem{definition}[theorem]{Definition}
\newtheorem{corollary}[theorem]{Corollary}

\title[Unconditional Uniqueness for the NLS]{Unconditional Uniqueness Results for the Nonlinear Schr\"{o}dinger Equation}
\begin{document}

\author[S.~Herr]{Sebastian Herr} \address{Universit\"{a}t Bielefeld,
  Fakult\"{a}t f\"{u}r Mathematik, Postfach 10 01 31, D-33501
  Bielefeld, Germany} \email{herr@math.uni-bielefeld.de}

\author[V.~Sohinger]{Vedran Sohinger} \address{University of Warwick, Mathematics Institute, Zeeman Building, Coventry CV4 7AL, United Kingdom}
\email{V.Sohinger@warwick.ac.uk}


\maketitle

\begin{abstract}
We study the problem of unconditional uniqueness of solutions to the cubic nonlinear Schr\"{o}dinger equation.
We introduce a new strategy to approach this problem on bounded domains, in particular on rectangular tori.

It is a known fact that solutions to the cubic NLS give rise to solutions of the Gross-Pitaevskii hierarchy, which is an infinite-dimensional system of linear equations. By using the uniqueness analysis of the Gross-Pitaevskii hierarchy, we obtain new unconditional uniqueness results for the cubic NLS on rectangular tori, which cover the full scaling-subcritical regime in high dimensions. In fact, we prove a more general result which is conditional on the domain.

In addition, we observe that well-posedness of the cubic NLS in Fourier-Lebesgue spaces implies unconditional uniqueness.
\end{abstract}

\section{Introduction}\label{sec:intro}

Consider the \emph{(defocusing, cubic) nonlinear Schr\"{o}dinger equation (NLS)}
\begin{equation}
\label{NLS}
\begin{cases}
\ii \partial_t \phi + \Delta \phi=|\phi|^2 \phi \\
\phi \big|_{t=0}=\phi_0
\end{cases}
\end{equation}
for given initial data $\phi_0 \in H^s(\R^d)$. The standard strategy of proving well-posedness of the initial value
problem is similar to that for the Picard-Lindel\"of theorem for
ODEs. One applies the contraction mapping principle in some suitable
Banach space $X\hookrightarrow C([-T,T],H^s(\R^d))$ to solve the
corresponding integral equation. Given rough initial data, i.e. if
$s<d/2$, typical choices for $X$ are mixed Lebesgue spaces (Strichartz
norms) \cite{Caz} or Fourier restriction spaces \cite{B93}. As a
consequence, this argument yields uniqueness of solutions in $X$, not unconditionally in $C([-T,T],H^s(\R^d))$. Tosio Kato \cite{K95,K96} was the first to address this problem, see Subsection \ref{subsec:prev} for further results and references. The purpose of this paper is to introduce new strategies which apply in the setting of certain bounded spatial domains.

\subsection{Setup of the problem and main results}\label{subsec:setup}

We are considering a spatial domain $\mathbb{D}$ which is either a $d$-dimensional rectangular torus $\mathbb{T}^d_g$ or, more generally, a smooth compact Riemannian manifold of dimension $d$ without boundary which is connected and orientable. In particular, on $\mathbb{D}$, we have a well-defined notion of a positive self-adjoint Laplacian $-\Delta \equiv -\Delta_{\mathbb{D}}$ and we can consider \eqref{NLS}
on the spatial domain $\mathbb{D}$. We address the question of \emph{unconditional uniqueness} of mild solutions of \eqref{NLS}. All of the results that we discuss are valid for arbitrary defocusing and focusing coupling constants in front of the cubic nonlinearity, but we set the coupling constant to be $1$ in \eqref{NLS} for simplicity of notation.

Before we give the precise definition of a mild solution of \eqref{NLS} we recall the definition of $L^2$-based Sobolev spaces on $\mathbb{D}$.
The operator $-\Delta$ has a discrete spectrum $0=\lambda_0 < \lambda_1 < \cdots < \lambda_k \rightarrow \infty$. 
Moreover, $\mathcal{H}_j$ denotes the eigenspace corresponding to eigenvalue $\lambda_j$,
\begin{equation*}
L^2(\mathbb{D}) \;=\; \bigoplus_{j=0}^{\infty} \mathcal{H}_j.
\end{equation*}
Given $k \in \mathbb{N}$, we let 
$
\Lambda_k \;\deq\; \{j \in \mathbb{N}_0:\,\sqrt{|\lambda_j|} \in [k-1,k) \}
$,
and $\chi_k$ be the orthogonal projection $\chi_k: L^2(\mathbb{D}) \rightarrow \bigoplus_{j \in \Lambda_k} \mathcal{H}_j$.
Given $s \in \mathbb{R}$, using $\langle x \rangle := \sqrt{1+|x|^2}$, we define the \emph{Sobolev space} $H^s \equiv H^s(\mathbb{D})$ via
\begin{equation}
\label{Sobolev space D}
\|f\|_{H^s}:=\bigg(\sum_{k=0}^{\infty} \langle \lambda_k \rangle^{s} \|\chi_k f\|_{L^2}^2\bigg)^{1/2}.
\end{equation}
Note that $\ee^{\ii t \Delta}$ acts as a unitary semigroup on $H^s$.
\begin{definition}[Mild solution of the NLS] 
\label{Mild Solution of the NLS}
Given $s \geq 0$, $T>0$, and $\phi_0 \in H^s(\mathbb{D})$, we say that
$\phi=\phi(t,x) $ is a mild solution of \eqref{NLS} in $H^s$ with initial data $\phi_0$ if $\phi \in L^{\infty}_{t\in [-T,T]} H^s_x \cap L^3([-T,T]\times \D)$ and the following conditions hold:
\begin{itemize}
\item[(i)]
For $t \in [-T,T]$ we have
\begin{equation}
\label{NLS mild solution}
\phi(x,t)=\ee^{\ii t\Delta} \phi_0(x)-\ii  \, \int_0^{t} \ee^{\ii(t-\tau) \Delta} |\phi|^2 \phi (x,\tau)\,\dd \tau\,. 
\end{equation}
\item[(ii)] There exists $R>0$ such that
\begin{equation}
\label{NLS mild solution 2}
\mathop{\sup}_{t \in {[-T,T]}}\|\phi(t)\|_{H^s(\mathbb{D})} \leq R\,. 
\end{equation}
\item[(iii)]
For $t \in [-T,T]$ we have
\begin{equation}
\label{NLS mild solution 3}
\|\phi(t)\|_{L^2(\mathbb{D})}= \|\phi_0\|_{L^2(\mathbb{D})}\,. 
\end{equation}
\end{itemize}
\end{definition}

\begin{remark}
\label{local_integrability}
Due to our assumption $\phi \in L^3([-T,T]\times \D)$ the nonlinearity
$|\phi|^2 \phi$ is integrable and defines a distribution on $[-T,T]
\times \mathbb{D}$. If we have the Sobolev embedding
$H^{\frac{d}{6}}(\D)\subset L^3(\D)$ on the domain $\D$ (e.g.\ on
rectangular tori), then $\phi \in L^3([-T,T]\times \D)$ is implied by
$\phi \in L^{\infty}_{t \in [-T,T]} H^s_x$ if $s\geq \frac{d}{6}$. By
the Sobolev embedding on $\D$ (see Proposition \ref{Sobolev_embedding}) it follows that we always have
\begin{equation}
\label{H_Beta}
|\phi|^2 \phi \in L^\infty_{t \in [-T,T]} H^{\beta}_x\, \text{ for } \beta<-\tfrac{d}{2},
\end{equation}
in particular, the integral in the right hand-side of \eqref{NLS mild solution} is taken in $H^{\beta}_x$, and it makes sense since we are integrating over a compact time interval. Note that, a posteriori, \eqref{NLS mild solution} tells us that the function defined by the integral belongs to $H^s_x$.
\end{remark}

For \emph{dyadic integers} $N=2^j$ we define
\begin{equation*}
P_N \deq \sum_{k \in [N,2N)} \chi_k\,.
\end{equation*}
Moreover we define $P_0 \deq \chi_0$. 
In particular, we have
\begin{equation}
\label{sum_P_N}
\sum_{N} P_N = \mathrm{Id}\,,
\end{equation}
where $\sum_{N}$ means that we are summing over all dyadic integers $N$ and $N=0$. 
By \eqref{Sobolev space D}, we have
\begin{equation}
\label{dyadic Sobolev space M}
\|f\|_{H^s}  \sim \Bigg(\sum_{n=0}^{\infty} \langle n \rangle^{2s} \|\chi_n f\|_{L^2}^2\Bigg)^{1/2} \sim \Bigg(\sum_{N} \langle N \rangle^{2s} \|P_N f\|_{L^2}^2\Bigg)^{1/2}\,.
\end{equation}
Furthermore we define the \emph{Besov space} $B^{s}_{2,1} \equiv B^{s}_{2,1}(\mathbb{D})$ 
\begin{equation}
\label{Besov space M}
\|f\|_{B^{s}_{2,1}}:=\sum_{N} \langle N \rangle^s \|P_N f\|_{L^2}\,.
\end{equation}
For all $\delta>0$ we have the inclusion
\begin{equation}
\label{Inclusion}
H^{s+\delta} \hookrightarrow B^{s}_{2,1} \hookrightarrow H^s\,.
\end{equation}

Motivated by the multi-linear analysis of \cite{Burq_Gerard_Tzvetkov_2005, Herr_2013,HTT} and the references therein we define the following.

\begin{definition}
\label{Admissible_domain}
We say that the domain $\mathbb{D}$ is admissible if there exist constants $\zeta_0 \equiv \zeta_0 (\mathbb{D}), \alpha_0 \equiv \alpha_0 (\mathbb{D}), \epsilon \equiv \epsilon (\mathbb{D}) >0$ such that for any fixed
\begin{equation}
\label{Condition_on_eta}
0 \leq \eta \leq \zeta_0\,, \quad \zeta> \zeta_0\,, \quad \alpha>\alpha_0
\end{equation}
the following estimates hold: For all $T \in [0,1]$ and  $u_j(t)=\ee^{\ii t\Delta} \phi_j$, $j=1,2,3$, 
\begin{align}
\label{Star1}
\|u_1 u_2 u_3 \|_{L^1_t B^{-\eta}_{2,1}([-T,T] \times \mathbb{D})} \lesssim{}& T^{\epsilon} \, \|\phi_1\|_{B^{-\eta}_{2,1}} \, \|\phi_2\|_{B^{\zeta}_{2,1}} \, \|\phi_3\|_{B^{\zeta}_{2,1}}\\
\label{Star2}
\|u_1 u_2 u_3 \|_{L^1_t B^{\zeta}_{2,1}([-T,T] \times \mathbb{D})} \lesssim{}& T^{\epsilon} \, \|\phi_1\|_{B_{2,1}^{\zeta}} \, \|\phi_2\|_{B_{2,1}^{\zeta}} \, \|\phi_3\|_{B_{2,1}^{\zeta}}\\
\label{Star3}
\|f_1f_2f_3\|_{B^{-\zeta_0}_{2,1}} \lesssim{}& \|f_1\|_{H^{\alpha}} \, \|f_2\|_{H^{\alpha}} \, \|f_3\|_{H^{\alpha}}\,.
\end{align}
For $\mathbb{D}$ an admissible domain we let 
\begin{equation}
\label{Definition_of_s_0}
s_0 \deq \max\{\zeta_0,\alpha_0,d/4\}\,.
\end{equation}
\end{definition}

\begin{remark}\label{rmk:star1impliesstar2}
On many domains \eqref{Star2} is a direct consequence of \eqref{Star1}, see e.g. Section \ref{torus_admissible} for details in the case of tori. 
\end{remark}

We now state our main results.
\begin{theorem}[Unconditional uniqueness for admissible domains]
\label{Main_theorem}
Let $\mathbb{D}$ be an admissible domain with parameter $s_0$ given by \eqref{Definition_of_s_0} and $s>s_0$. For any $\phi_0\in H^s$ there is a unique mild solution of \eqref{NLS} in $H^s$.
\end{theorem}

\begin{remark}
We refer to the uniqueness result in Theorem \ref{Main_theorem} as \textbf{unconditional} in the sense that we are only assuming the control of the $H^s$ norm of the solution as in Definition \ref{Mild Solution of the NLS} above. In particular, we do not assume that the solution belongs to some Strichartz space. 
\end{remark}

We can view Theorem \ref{Main_theorem} as a result that is \emph{conditional} on the domain. The condition is that $\mathbb{D}$ is admissible in the sense of Definition \ref{Admissible_domain} above. Once this assumption is satisfied, the uniqueness analysis follows by the general argument, which consists of comparing mild solutions of \eqref{NLS} to those of the \emph{Gross-Pitaevskii (GP) hierarchy on $\mathbb{D}$}, which is defined in \eqref{GP hierarchy} below. In particular, we use the observation that tensor products of mild solutions of the NLS are mild solutions of the GP hierarchy, which is a well-known fact (for classical solutions). In light of this observation, one typically considers the GP hierarchy as a generalization of the NLS and proves analogues of results known for the NLS in the context of the hierarchy. This point of view was first taken in \cite{CP1}.
Our work is the first instance in which we can use knowledge from the GP hierarchy to answer nontrivial questions about the NLS.
 
In Section \ref{torus_admissible}, we show that rectangular tori of dimension $d \geq 2$ are admissible in the sense of Definition \ref{Admissible_domain}. We now state the explicit unconditional uniqueness result that we obtain from Proposition \ref{Torus_admissibility}.
\begin{corollary}
\label{unconditional_uniqueness_torus}
Solutions to \eqref{NLS} are unconditionally unique in $H^s(\mathbb{T}^d_g)$ provided that
\begin{equation}
\label{unconditional_uniqueness_torus_s}
\begin{cases}
s>\frac{7}{12}\,\,&\mbox{for } d=2
\\
s>\frac{4}{5}\,\,&\mbox{for } d=3
\\ 
s>\frac{d}{2}-1\,\,&\mbox{for } d \geq 4\,.
\end{cases}
\end{equation}
\end{corollary}

\begin{remark}
\label{unconditional_uniqueness_torus_remark}
The required regularity in Corollary \ref{unconditional_uniqueness_torus} is below the energy class for $d=2,3$. For $d \geq 4$, \eqref{NLS} is energy-supercritical and \eqref{unconditional_uniqueness_torus_s} covers the full scaling-subcritical range.
\end{remark}

The above method also applies in one dimension, i.e.\ in the case when $\mathbb{D}=\mathbb{T}$ and yields unconditional uniqueness in $H^{s}(\T)$ for $s>\frac13$. However, we can obtain a stronger result by a different method.

\begin{theorem}\label{thm:uu-1d}
Let $p>3$. For any $\phi_0\in L^2(\T)$ there is a unique mild solution $\phi$ of \eqref{NLS} in $L^2(\T)$ with the property $\phi \in L^p([-T,T]\times \T)$.
\end{theorem}
Note that the above uniqueness result is in an almost
scaling-invariant class and, by the Sobolev embedding, implies
unconditional uniqueness in $H^{s}(\T)$ for $s>\frac16$. We
emphazise that this is strictly weaker than the result obtained
in \cite{GuoKwonOh}, where unconditional uniqueness is proved in
$H^{\frac16}(\T)$ by a normal form method with infinitely many iterations. In comparison, the proof of Theorem \ref{thm:uu-1d} is an easy consequence of an estimate in \cite{GH}.

\subsection{Previously known results}\label{subsec:prev}
The problem of unconditional uniqueness of solutions to the NLS has been extensively studied in the literature. 
Using Sobolev embedding and Gronwall's inequality, one can easily obtain an unconditional uniqueness result the NLS when $s>d/2$. Note that, for $d \geq 2$, this is already above the energy space. 
We henceforth consider only the nontrivial regime $s<d/2$.
The first such result on $\mathbb{R}^d$ is that of Kato \cite{K95,K96}, with subsequent extensions and improvements in dimension $d \geq 2$ obtained in \cite{Cazenave_Han_Fang,FurioliTerraneo,FurioliTerraneoPlanchon,HanFang,Lu_Xu,Rogers,Tao,Win_Tsutsumi}, we refer to these references for precise results. The inherent losses of derivatives in the Strichartz estimate do not allow one to adapt these arguments to compact domains. The first unconditional uniqueness result on the one-dimensional torus was obtained in \cite{GuoKwonOh} by the use of normal form techniques and it applies to data in $H^{\frac16}(\T)$, in agreement with the result from \cite{K95,K96} in $H^{\frac16}(\R)$.  In fact, in the case $\D=\R$, the result in \cite{K95,K96} gives uniqueness of mild solutions $\phi \in L^\infty_t L^2_x\cap L^{12}_tL^3_x$. An extension of the techniques of \cite{GuoKwonOh} to higher dimensions has been announced by \cite{Kishimoto}. The announced regime agrees with the regime from Corollary \ref{unconditional_uniqueness_torus} when $d \geq 6$.
We note that the unconditional uniqueness problem was also studied for other dispersive models on $\mathbb{R}^d$ \cite{Bulut_Czubak_Li_Pavlovic_Zhang,Farah,FurioliTerraneoPlanchon,Masmoudi_Nakanishi,Masmoudi_Planchon,Planchon}.

In the method of deriving the NLS from many-body quantum dynamics developed by Spohn \cite{2S}, a central role is played by the GP hierarchy \eqref{GP hierarchy}. The final step in this approach of deriving the NLS typically consists of showing uniqueness of solutions to \eqref{GP hierarchy}. 
This is a non-trivial problem due to the large number of terms that one has to consider. Several methods have been introduced to deal with this issue. They are a Feynman graph expansion \cite{ESY2}, the \emph{boardgame argument} \cite{KM}, or the quantum de Finetti theorem \cite{ChHaPavSei}. This has been an extensively studied problem.
We refer the reader to the introduction of \cite{HerrSohinger,VS2} for further references and for a more detailed discussion of the methods.
In a recent preprint \cite{Ammari_Liard_Rouffort}, the authors study connections between solutions of an initial value problem and an associated hierarchy in order to obtain results about the uniqueness of the hierarchy. This is the opposite direction from the one that we are taking in the proof of Theorem \ref{Main_theorem}.

Our second method to obtain Theorem \ref{thm:uu-1d} is based on
multi-linear estimates in restriction norms based on Fourier-Lebesgue
spaces and the Hausdorff-Young inequality, we refer to Section
\ref{sec:1d} for more details on this. While the implied result
on unconditional uniqueness in $H^{s}(\T)$ for $s>\frac16$ is not new,
the general observation here is that the analysis of dispersive PDEs on tori in Fourier-Lebesgue spaces yields unconditional uniqueness results in $L^2$-based spaces.

\subsection{Organization of the paper}\label{subsec:orga}
In Section \ref{The Gross-Pitaevskii hierarchy}, we recall the definition of the Gross-Pitaevskii hierarchy and we prove the general result given by Theorem \ref{Main_theorem}.
In Section \ref{torus_admissible}, we apply Theorem \ref{Main_theorem} to the setting of rectangular tori and show the explicit unconditional uniqueness result stated in Corollary \ref{unconditional_uniqueness_torus} above. Section \ref{sec:1d} is devoted to the proof of Theorem \ref{thm:uu-1d}. In Appendix \ref{Appendix A}, we prove the multi-linear estimate given by Proposition \ref{Product estimate proposition}, which is used in the proof of Theorem \ref{Main_theorem}. In doing so, we recall several general Sobolev embedding results. In Appendix \ref{appendix-bern}, we give a proof of the Bernstein inequality on rectangular tori, which is stated in Lemma \ref{Bernstein_inequality_torus}. In Appendix \ref{sec:lin-est}, we give the proof of a key linear estimate used in the proof of Theorem \ref{thm:uu-1d}.

\section{The Gross-Pitaevskii hierarchy}
\label{The Gross-Pitaevskii hierarchy}
In this section, we study the Gross-Pitaevskii hierarchy and prove Theorem \ref{Main_theorem}.
The required notation and definitions are given in Subsection \ref{Definition of the hierarchy}. In Subsection \ref{Factorized solutions of the Gross-Pitaevskii hierarchy}, we study tensor products of mild solutions of \eqref{NLS} and show that they solve the Gross-Pitaevskii hierarchy in an appropriate sense. Finally, we give the proof of Theorem \ref{Main_theorem} in Subsection \ref{The uniqueness analysis}.

\subsection{Definition of the hierarchy}
\label{Definition of the hierarchy}


Given $k \in \mathbb{N}$ we call functions $\sigma^{(k)}:\mathbb{D}^k \times \mathbb{D}^k \rightarrow \mathbb{C}$ \emph{density matrices of order $k$}, and we denote this class by $\mathcal{M}_k$. In \eqref{GP hierarchy}, for each fixed $t$, $\gamma^{(k)}(t) \in \mathcal{M}_k$. Given $\sigma^{(k+1)} \in \mathcal{M}_{k+1}$ we formally define $B_{k+1} \sigma^{(k+1)} \in \mathcal{M}_{k+1}$ by
\begin{equation}
\label{Collision_Operator_1}
B_{k+1} \sigma^{(k+1)} \deq \sum_{j=1}^{k} B_{j,k+1}(\sigma^{(k+1)})\,,
\end{equation}
where, for $1 \leq j \leq k$ $(\vec{x}_k;\vec{x}_k') \in \mathbb{D}^k \times \mathbb{D}^k$ we define
\begin{align}
\notag
&B_{j,k+1}(\sigma^{(k+1)}) (\vec{x}_k;\vec{x}_k') \deq 
\\
\label{Collision_Operator_2}
&\int_{\mathbb{D}} \Big(\delta(x_j-x_{k+1})-\delta(x_j'-x_{k+1})\Big) \sigma^{(k+1)}(\vec{x}_k,x_{k+1};\vec{x}_k',x_{k+1}) \,\dd x_{k+1}\,.
\end{align}
$B_{k+1}$ is called the \emph{collision operator}.
At this point we view \eqref{Collision_Operator_1} and \eqref{Collision_Operator_2} as formal definitions on the level of density matrices \footnote{In \eqref{Collision_Operator_2} we are just formally restricting the function $\sigma^{(k+1)}$ to the set where $x_{k+1}=x_{k+1}'=x_j$ and $x'_j=x_{k+1}=x'_j$ respectively.}. 
For our purposes we are going to look at $\sigma^{(k+1)}$ of the form
\begin{equation*}
\sigma^{(k+1)}(\vec{x}_{k+1};\vec{x}'_{k+1}) = {\prod_{\ell=1}^{k+1}} f_{\ell} (x_\ell) g_\ell (x_\ell')\,,
\end{equation*}
for some $f_1,\ldots,f_k,g_1,\ldots,g_k \in L^2(\mathbb{D})$ or superpositions thereof. All of the objects involving the collision operator that we will consider will be well-defined by using this a priori formal definition, see Lemmas \ref{Mild factorized solutions} and \ref{sigma_bound} below. 
In principle it is possible to impose additional regularity assumptions on $\sigma^{(k+1)}$ under which these operations can be written in terms of an approximate identity (see \cite[Theorem 1(iv)]{EESY}). We do not take this approach here.

The \emph{(defocusing, cubic) Gross-Pitaevskii (GP) hierarchy} on $\mathbb{D}$ is given by
\begin{equation}
\label{GP hierarchy}
\begin{cases}
\ii \partial_t \gamma^{(k)} + (\Delta_{\vec{x}_k}-\Delta_{\vec{x}'_k}) \gamma^{(k)}=B_{k+1} \gamma^{(k+1)}\\
\gamma^{(k)} \big|_{t=0}=\gamma_0^{(k)}\,.
\end{cases}
\end{equation}
Here, $\gamma^{(k)}$ are time-dependent density matrices of order $k$. We always assume that the right-hand side of the first equation in \eqref{GP hierarchy} is well-defined in the sense described above.

On density matrices of order $k$ we define the Sobolev differentiation operators 
\begin{equation*}
S^{(k,\alpha)} := \prod_{j=1}^{k} \langle \Delta_{x_j} {\rangle^{\alpha/2}} \prod_{j=1}^{k} \langle \Delta_{x_j'} \rangle^{\alpha/2}
\end{equation*}
and the free Schr\"{o}dinger evolution operators
\begin{equation}
\label{U^k}
  \mathcal{U}^{(k)}(t) \,\sigma^{(k)}:=e^{it \sum_{j=1}^{k} \Delta_{x_j}} \sigma^{(k)} e^{-it\sum_{j=1}^{k} \Delta_{x'_j}}
\end{equation}
in the standard way.
We note that \eqref{U^k} is well-defined if we view $\sigma^{(k)} \equiv \sigma^{(k)}(\vec{x}_k,\vec{x}'_k)$ as the kernel of an operator on $L^2(\mathbb{D}^k \times \mathbb{D}^k)$.

\begin{definition}[The class $\mathfrak{H}^s$]
\label{mathfrakHs}
Given $s \geq 0$, $\mathfrak{H}^{s} \equiv \mathfrak{H}^s(\mathbb{D})$ denotes the set of all sequences $(\gamma^{(k)})_k$ of density matrices such that, for each $k \in \mathbb{N}$ the following properties hold:
\begin{itemize}
\item[(i)]  $\gamma^{(k)} \in L^2_{sym}(\mathbb{D}^k \times \mathbb{D}^k)$ in the sense that $\gamma^{(k)} \in L^2(\mathbb{D}^k \times \mathbb{D}^k)$ and
\begin{equation*}\gamma^{(k)}(x_{\pi(1)},\ldots,x_{\pi(k)};x'_{\pi(1)},\ldots,x'_{\pi(k)}) =\gamma^{(k)}(x_1,\ldots,x_k;x'_1,\ldots,x'_k)
\end{equation*}
for all permutations $\pi \in \mathcal{S}^k$. 
Moreover, 
$$\gamma^{(k)}(\vec{x}_k;\vec{x}'_k)=\overline{\gamma^{(k)}(\vec{x}'_k;\vec{x}_k)}$$ 
for all $(\vec{x}_k,\vec{x}'_k)$ in $\mathbb{D}^k \times \mathbb{D}^k$.
\item[(ii)] $S^{(k,s)}\gamma^{(k)} \in \mathcal{L}^1_k$, i.e.\ $S^{(k,s)}\gamma^{(k)}$ belongs to the trace class on $L^2(\mathbb{D}^k \times \mathbb{D}^k)$, whose norm is given by $\|\mathcal{T}\|_{\mathcal{L}^1_k} := \tr |\mathcal{T}|$, where $|\mathcal{T}| := \sqrt{\mathcal{T}^*\,\mathcal{T}}$.
\end{itemize}
We abbreviate (i) and (ii) as $\gamma^{(k)} \in \mathfrak{H}^s_k= \mathfrak{H}^s_k(\mathbb{D})$ and we define the norm on this space as
\begin{equation*}
\|\gamma^{(k)}\|_{\mathfrak{H}^s_k} := \tr |S^{(k,s)}\gamma^{(k)}|\,.
\end{equation*}
\end{definition}

\begin{definition}[Mild solution of the Gross-Pitaevskii hierarchy]
\label{Mild solution of the Gross-Pitaevskii hierarchy} Given $s > \frac{d}{4}$, $T>0$, and 
$\big(\gamma_0^{(k)}\big)_k \in \mathfrak{H}^{s}$,
we say that $\big(\gamma^{(k)}\big)_k=\big(\gamma^{(k)}(\vec{x}_k;\vec{x}'_k,t)\big)_k$ is a mild solution of \eqref{GP hierarchy} in $\mathfrak{H}^{s}$ with initial data $\big(\gamma_0^{(k)}\big)_k\,$
 if the following conditions hold:
\begin{itemize}
\item[(i)]
For all $k \in \mathbb{N}$ we have
\begin{equation}
\label{GP hierarchy mild solution 2}
\gamma^{(k)} \in L^{\infty}_{t \in [-T,T]} \mathfrak{H}_k^s\,.
\end{equation}
\item[(ii)] 
For all $k \in \mathbb{N}$ and $t \in [-T,T]$ we have
\begin{equation}
\label{GP hierarchy mild solution}
\gamma^{(k)}(t)=\mathcal{U}^{(k)}(t) \gamma_0^{(k)}-\ii \, \int_0^t \mathcal{U}^{(k)}(t-s) B_{k+1} \gamma^{(k+1)} (s)\,\dd s\,.
\end{equation}
\item[(iii)] There exists a constant $\mathcal{R}>0$ such that 
\begin{equation}
\label{GP hierarchy mild solution 3}
\tr \big(|S^{(k,s)}\gamma^{(k)}(t)|\big) \leq \mathcal{R}^{2k}
\end{equation}
for all $k \in \mathbb{N}$ and $t \in [-T,T]$.
\end{itemize}
\end{definition}

\begin{remark}
We note that, for $t \in [-T,T]$ equality in \eqref{NLS mild solution} is assumed to hold in $H^s$, whereas in \eqref{GP hierarchy mild solution} it is assumed to hold in $\mathfrak{H}^s_k$.
\end{remark}

\subsection{Factorized solutions of the Gross-Pitaevskii hierarchy}
\label{Factorized solutions of the Gross-Pitaevskii hierarchy}
We first note a general relationship between mild solutions of \eqref{NLS} and \eqref{GP hierarchy}.
\begin{lemma} [Mild factorized solutions of the Gross-Pitaevskii hierarchy]
\label{Mild factorized solutions} 
Let $s \geq 0$, $T>0$ and $\phi_0 \in H^s(\mathbb{D})$ be given.
Suppose that $\phi$ is a mild solution of \eqref{NLS} in $H^s$ with initial data $\phi_0$. Then
\begin{equation}
\label{factorized_solution}
\gamma^{(k)}(\vec{x}_k;\vec{x}'_k,t):=|\phi \rangle \langle \phi|^{\otimes k} (\vec{x}_k;\vec{x}'_k,t)
\end{equation}
is a mild solution of \eqref{GP hierarchy} in $\mathfrak{H}^{s}$ with initial data $\big(|\phi_0 \rangle \langle \phi_0|^{\otimes k}\big)_k$.
\end{lemma}
It is a known fact that tensor products of classical solutions to the NLS give us classical solutions of the Gross-Pitaevskii hierarchy. Lemma \ref{Mild factorized solutions} tells us that the analogous claim is true when we consider mild solutions. Note that we cannot deduce Lemma \ref{Mild factorized solutions} by a density argument from classical solutions, because this would in particular rely on unconditional uniqueness of solutions to \eqref{NLS}, which is what we are proving. Instead, we have to argue directly and verify the assumptions of Definition \ref{Mild solution of the Gross-Pitaevskii hierarchy} above.
Before proving Lemma \ref{Mild factorized solutions} we note a general algebraic fact that we will use in the proof.

\begin{lemma}
\label{Useful_algebraic_fact}
Given $m \in \mathbb{N}$, numbers $F_1, \ldots, F_m \in \mathbb{C}$ and locally integrable functions $G_1, \cdots, G_m :\mathbb{R} \rightarrow \mathbb{C}$, we have for all $t \in \mathbb{R}$
\begin{align}
\notag
&\prod_{r=1}^{m} \Big(F_r+\int_0^t G_r(t_r)\,\dd t_r\Big)=
\\
&\mathop{\sum_{\mathcal{S}_1 \sqcup \mathcal{S}_2=}}_{\{1,2,\ldots,m\}} \Big(\prod_{r \in \mathcal{S}_1} F_r\Big) \, \Bigg(\sum_{r \in \mathcal{S}_2} \int_0^{t} G_r(\tau_r) \, \Big(\prod_{\tilde{r} \in \mathcal{S}_2 \setminus \{r\}} \int_0^{\tau_r} G_{\tilde{r}}(\tau_{\tilde{r}}) \,\dd \tau_{\tilde{r}}\Big) \,\dd \tau_r\Bigg)\,.
\end{align}
\end{lemma}

\begin{proof}[Proof of Lemma \ref{Useful_algebraic_fact}]
Let $\mathcal{S}_2 \subset \{1,2,\ldots,m\}$ be nonempty. We note that
\begin{equation*}
\Bigg(\prod_{r \in \mathcal{S}_2} \int_0^{t} G_r(\tau_r)\,d\tau_r\Bigg)=\Bigg(\sum_{r \in \mathcal{S}_2} \int_0^{t} G_r(\tau_r) \, \Big(\prod_{\tilde{r} \in \mathcal{S}_2 \setminus \{r\}} \int_0^{\tau_r} G_{\tilde{r}}(\tau_{\tilde{r}}) \,\dd \tau_{\tilde{r}}\Big) \,\dd \tau_r\Bigg)\,.
\end{equation*}
The above identity follows by using Fubini's theorem for the integral on the left-hand side and by decomposing the region of integration $[0,t]^{{|\mathcal{S}_2|}}$ in terms of the maximal component $\tau_r$ with $r \in \mathcal{S}_2$. 
The claim now follows by expanding the product in the expression on the left-hand side of Lemma \ref{Useful_algebraic_fact}.
\end{proof}
\begin{proof}[Proof of Lemma \ref{Mild factorized solutions}]
From \eqref{NLS mild solution 2}, we deduce the condition \eqref{GP hierarchy mild solution 2} of being a mild solution, as well as \eqref{GP hierarchy mild solution 3} with $\mathcal{R}=R$.
We hence have to verify \eqref{GP hierarchy mild solution} for fixed $k \in \mathbb{N}$ and $t \in [-T,T]$. 
By using \eqref{NLS mild solution}, it follows that
\begin{align}
& \gamma^{(k)}(\vec{x}_k;\vec{x}'_k,t)
=\prod_{j=1}^{k} \Bigg[\Big(\ee^{\ii t \Delta} \phi_0(x_j)+ \int_0^{t} -\ii  \, \ee^{\ii(t-\tau_j) \Delta} |\phi|^2 \phi (x_j,\tau_j)\,\dd \tau_j \Big) \notag\\
&\times \Big(\overline{ \ee^{\ii t \Delta} \phi_0(x'_j)}+ \int_0^{t}\overline{ -\ii \, \ee^{\ii(t-\tau'_j) \Delta} |\phi|^2 \phi (x'_j,\tau'_j)}\,\dd \tau'_j \Big)\Bigg]\,. \label{LHS k}
\end{align}
By \eqref{U^k}, it follows that
\begin{equation}
\label{RHS k 1}
\mathcal{U}^{(k)}(t) \gamma_0^{(k)}(\vec{x}_k;\vec{x}'_k)=\prod_{j=1}^{k} \ee^{\ii t \Delta} \phi_0(x_j) \, \overline{\ee^{\ii t \Delta} \phi_0(x'_j)}.
\end{equation}
We compute
\begin{align*}
&-\ii \, \int_0^t \mathcal{U}^{(k)}(t-\tau) \,B_{k+1}  \gamma^{(k+1)} (\vec{x}_k;\vec{x}'_k,\tau)\,\dd \tau\\
=&  -\ii \, \int_0^{t} \mathcal{U}^{(k)}(t-\tau) \, \sum_{j=1}^{k} \,  \Big[\Big(|\phi|^2 \phi (x_j,\tau) \, \overline{\phi(x_j',\tau)} -\phi(x_j,\tau) \, |\phi|^2 \bar{\phi}(x'_j,\tau) \Big) \, \mathop{\prod_{{\ell}=1}^{k}}_{\ell \neq j} \phi(x_{\ell},\tau) \, \overline{\phi(x'_{\ell},\tau)} \Big] \,\dd \tau  \\
=& -\ii \, \int_0^{t} \sum_{j=1}^{k} \Bigg[\Big( \ee^{\ii(t-\tau) \Delta} |\phi|^2 \phi (x_j,\tau) \, \overline{\ee^{\ii(t-\tau)\Delta} \phi(x_j',\tau)}-\ee^{\ii(t-\tau)\Delta}\phi(x_j,\tau) \, \overline{\ee^{\ii(t-\tau) \Delta} |\phi|^2 \phi (x'_j,\tau)} \Big) \\
&\times \mathop{\prod_{\ell=1}^{k}}_{\ell \neq j} \ee^{\ii(t-\tau)\Delta}\phi(x_{\ell},\tau) \, \overline{\ee^{\ii(t-\tau)\Delta}\phi(x'_{\ell},\tau)}
\Bigg] \,\dd \tau.
\end{align*}
We again use \eqref{NLS mild solution} and the semigroup property to deduce that the above expression equals
\begin{align}
& \int_0^{t} \sum_{j=1}^{k} \Bigg\{ \Bigg[ \Big(-\ii \, \ee^{\ii(t-\tau) \Delta} |\phi|^2 \phi (x_j,\tau) \Big) \, \Big(\overline{\ee^{\ii t \Delta} \phi_0(x'_j)}+\int_0^{\tau} \overline{-\ii \, \ee^{\ii(t-\tau'_j) \Delta} |\phi|^2 \phi(x'_j,\tau'_j)}\,\dd \tau'_j\Big) \notag \\
&+\Big(\ee^{\ii t \Delta} \phi_0(x_j) + \int_0^{\tau} -\ii\, \ee^{\ii(t-\tau_j) \Delta} |\phi|^2 \phi(x_j,\tau_j)\,\dd \tau_j\Big) \,  \overline{-\ii \, \ee^{\ii(t-\tau) \Delta} |\phi|^2 \phi (x'_j,\tau) }\Bigg] \notag\\
&\times \mathop{\prod_{\ell=1}^{k}}_{\ell \neq j} \Bigg[\Big(\ee^{\ii t \Delta}\phi_0(x_{\ell}) + \int_0^{\tau} -\ii\, \ee^{\ii(t-\tau_{\ell}) \Delta} |\phi|^2 \phi(x_{\ell},\tau_{\ell})\,\dd \tau_{\ell}\Big) \notag \\
& \Big( \overline{\ee^{\ii t \Delta}\phi_0(x'_{\ell})} + \int_0^{\tau}\overline{-\ii\, \ee^{\ii(t-\tau'_{\ell}) \Delta} |\phi|^2 \phi(x_{\ell},\tau'_{\ell})}\,\dd \tau'_{\ell}\Big) \Bigg] \Bigg\}\,\dd \tau \label{RHS k 2}\,.
\end{align}

The claim now follows from \eqref{LHS k}, \eqref{RHS k 1}, \eqref{RHS k 2} if we apply Lemma \ref{Useful_algebraic_fact} with $m=2k$ and functions $F_r,G_r$ given by
\begin{equation}
\label{Fr_Gr}
\begin{cases}
F_r:=\ee^{\ii t \Delta} \phi_0(x_r)\,\,\mbox{for}\,\,r \in\{1,2,\ldots,k\}\\
\\
F_r:=\overline{\ee^{\ii t \Delta} \phi_0(x'_{r-k})}\,\,\mbox{for}\,\,r \in\{k+1,k+2,\ldots,2k\}\\
\\
G_r(\tau):=-\ii  \, e^{\ii(t-\tau) \Delta} |\phi|^2 \phi (x_r,\tau)\,\,\mbox{for}\,\,r \in\{1,2,\ldots,k\}\\
\\
G_r(\tau):=\overline{-\ii  \, e^{\ii(t-\tau) \Delta} |\phi|^2 \phi (x'_{r-k},\tau)}\,\,\mbox{for}\,\,r \in\{k+1,k+2,\ldots,2k\}\,.
\end{cases}
\end{equation}
We note that we are applying Lemma \ref{Useful_algebraic_fact} to a time integral in functions that take values in the Sobolev space $H^{\beta}_x$, for $\beta$ as in \eqref{H_Beta} above. In particular, applying Lemma \ref{Useful_algebraic_fact} at this step is well-justified.
\end{proof}

\begin{remark}
We note that, in the proof of Lemma \ref{Mild factorized solutions} , we did not use condition \eqref{NLS mild solution 3} of being a mild solution of the NLS. This condition will be used crucially at several later points, see \eqref{modulus one phase}, \eqref{Case 1 lower bound on v}, and \eqref{Case 2 lower bound on v} below.
\end{remark}

\subsection{The uniqueness analysis: Proof of Theorem \ref{Main_theorem}}
\label{The uniqueness analysis}

Throughout this subsection we consider $s>s_0$, for $s_0$ as in \eqref{Definition_of_s_0}. 
It suffices to consider $s <\frac{d}{2}$.
Moreover, we fix $\phi_0 \in H^s(\mathbb{D})$ and consider two mild solutions $\phi^{(1)},\phi^{(2)}$ of \eqref{NLS} in $H^s$ with initial data $\phi_0$ as in Definition \ref{Mild Solution of the NLS}. Our goal is to show that $\phi^{(1)}=\phi^{(2)}$. Before proving Theorem \ref{Main_theorem} we first note a uniqueness result for the \emph{factorized solutions} \eqref{factorized_solution} of \eqref{GP hierarchy}.

\begin{proposition}
\label{factorized_GP_uniqueness}
Let $\phi^{(1)},\phi^{(2)}$ be mild solutions of \eqref{NLS} in $H^s$ with initial data $\phi_0 \in H^s$. We have
\begin{equation}
\label{factorized_GP_uniqueness2}
\big|\phi^{(1)} \rangle \langle \phi^{(1)}\big|^{\otimes k} =\big|\phi^{(2)} \rangle \langle \phi^{(2)}\big|^{\otimes k}\,,
\end{equation}
for all $k \in \mathbb{N}$.
\end{proposition}

We also note a multi-linear estimate which follows from Sobolev embedding.
\begin{proposition}
\label{Product estimate proposition}
Suppose that $\varrho \in (\frac{d}{4},\frac{d}{2})$ is given. Then we have for all $\delta>0$ 
\begin{equation*}
\|f_1 f_2 f_3\|_{H^{3 \varrho-d}} \lesssim_{\delta} \|f_1\|_{H^{\varrho+\delta}} \|f_2\|_{H^{\varrho+\delta}} \|f_3\|_{H^{\varrho+\delta}}\,.
\end{equation*}
\end{proposition}

Assuming Propositions \ref{factorized_GP_uniqueness} and \ref{Product estimate proposition} for now, we give the proof of Theorem \ref{Main_theorem}.

\begin{proof}[Proof of Theorem \ref{Main_theorem}]
By condition \eqref{NLS mild solution 3} of being a mild solution of \eqref{NLS}, we note that $\phi_0 = 0$ implies that $\phi^{(1)}(\,,t)=\phi^{(2)}(\,,t)=0$ for all $t \in [-T,T]$. Hence, we need to consider the case when $\phi_0 \neq 0$.

We note that \eqref{factorized_GP_uniqueness2} for $k=1$ implies that for all $x,x' \in \mathbb{D}$ we have

\begin{equation}
\notag
\phi^{(1)}(x,t) \, \overline{\phi^{(1)}(x',t)}=\phi^{(2)}(x,t) \, \overline{\phi^{(2)}(x',t)}
\end{equation}
We multiply both sides of the above identity with $\phi^{(2)}(x',t)$ and we integrate in $x'$ to deduce that
\begin{equation}
\notag
\phi^{(2)}(x,t) \, \|\phi^{(2)}(t)\|_{L^2_x}^2 = \phi^{(1)}(x,t) \, \langle \phi^{(2)}(t), \phi^{(1)}(t) \rangle_{L^2_x}\,.
\end{equation}
Hence, by \eqref{NLS mild solution 3} for $\phi^{(2)}$, we have
\begin{equation}
\label{phase a}
\phi^{(2)}(x,t)=a(t) \, \phi^{(1)}(x,t)\,,
\end{equation}
where
\begin{equation*}
a(t) \deq \frac{\langle \phi^{(2)}(t), \phi^{(1)}(t) \rangle_{L^2_x}}{\|\phi_0\|_{L^2_x}^2}\,.
\end{equation*}
Setting $t=0$, and recalling that $\phi^{(1)}(0)=\phi^{(2)}(0)=\phi_0$, it follows that
\begin{equation}
\label{a(0)}
a(0)=1.
\end{equation}
Substituting \eqref{phase a} into \eqref{NLS mild solution 3}, it follows that, for all $t \in [-T,T]$ we have
\begin{equation}
\label{modulus one phase}
|a(t)|=1.
\end{equation}
Let us note that \eqref{phase a} with $a$ satisfying conditions \eqref{a(0)} and \eqref{modulus one phase} is all that we can deduce on the level of solutions to \eqref{NLS} from Proposition \ref{factorized_GP_uniqueness}.
Namely, any $\phi^{(1)}$ and $\phi^{(2)}$ satisfying the above conditions will also satisfy \eqref{factorized_GP_uniqueness2}. 

Our goal now is to show that $a=1$. This claim holds trivially if $\phi^{(1)}$ and $\phi^{(2)}$ are classical solutions and if $a$ is differentiable in $t$. In our case, we have to work a bit more. We start from \eqref{NLS mild solution}. Using \eqref{phase a} and \eqref{modulus one phase}, we can rewrite \eqref{NLS mild solution} for $\phi^{(2)}$ as 
\begin{equation*}
a(t) \, \phi^{(1)}(x,t)= \ee^{\ii t\Delta} \phi_0 - \ii \, \int_0^{t}\ee^{\ii(t-\tau)\Delta} |\phi^{(1)}|^2 \phi^{(1)} (x,\tau) \, a(\tau) \,\dd \tau\,.
\end{equation*}
Taking differences of the above equality with \eqref{NLS mild solution} for $\phi^{(1)}$ we get
\begin{equation}
\label{difference equation}
(a(t)-1) \, \phi^{(1)}(x,t)=-\ii \, \int_0^{t} \ee^{\ii(t-\tau)\Delta} |\phi^{(1)}|^2\phi^{(1)}(x,\tau) \, (a(\tau)-1) \,\dd \tau\,.
\end{equation}
By recalling \eqref{Definition_of_s_0} and the assumption $s>s_0$, we can find $\varrho \in (\frac{d}{4},\frac{d}{2})$ such that $\varrho<s$. We fix such a parameter $\varrho$ for the remainder of the proof. Let us now estimate both sides of \eqref{difference equation} in the $H^{3\varrho-d}$ norm. We have
\begin{align*}
\notag
&\big\|(a(t)-1) \, \phi^{(1)}(t) \big\|_{H^{3\varrho-d}}=\bigg\|-\ii \, \int_0^{t} \ee^{\ii(t-\tau)\Delta} |\phi^{(1)}|^2\phi^{(1)}(\,,\tau) \, (1-a(\tau)) \,\dd \tau \bigg\|_{H^{3\varrho-d}}\\
&
\leq
\int_0^{t} \big\||\phi^{(1)}|^2\phi^{(1)}(\tau)\|_{H^{3\varrho-d}} \, |a(\tau)-1| \,\dd \tau
\lesssim_{\delta} \int_0^{t} \|\phi^{(1)}(t)\|_{H^{\varrho+\delta}}^3 \, |a(\tau)-1| \,\dd \tau\,.
\end{align*}
In the last step we applied Proposition \ref{Product estimate proposition}, which is possible by the assumptions on $\varrho$.
Hence, by \eqref{NLS mild solution 2} for $\phi^{(1)}$, we have for $\delta>0$ sufficiently small
\begin{align}
\notag
|a(t)-1| \, \|\phi^{(1)}(t)\|_{H^{3\varrho-d}} \lesssim \int_0^{t} \|\phi^{(1)}(\tau)\|_{H^s}^3 \, |a(\tau)-1| \,\dd \tau 
\\
\label{Gronwall 1}
\lesssim 
\int_0^{t} R^3 \, |a(\tau)-1| \,\dd \tau\,. 
\end{align}

The fact that $a =1$ follows from \eqref{Gronwall 1} by an application of Gronwall's inequality if we show that $\|\phi^{(1)}(x,t)\|_{H^{3s_0(d)-d}_x}$ is bounded from below. In order to do this, we need to consider two cases.

\begin{itemize}
\item[\textbf{Case 1:}] $3 \varrho - d \geq 0$.
\\
In this case we note that 
\begin{equation}
\label{Case 1 lower bound on v}
\|\phi^{(1)}(t)\|_{H^{3\varrho-d}} \geq \|\phi^{(1)}(t)\|_{L^2}=\|\phi_0\|_{L^2_x}\,,
\end{equation}
by using \eqref{NLS mild solution 3}.

\item[\textbf{Case 2:}] $3 \varrho -d<0$.
\\
In this case we interpolate. Namely, we know that $3 \varrho-d<0<s$, hence there exists $\theta \in (0,1)$ such that 
\begin{equation}
\notag
\|\phi^{(1)}(t)\|_{L^2} \leq \|\phi^{(1)}(t)\|_{H^{3\varrho-d}}^{\theta} \, \|\phi^{(1)}(t)\|_{H^s}^{1-\theta}\,.
\end{equation}
\end{itemize}
Using \eqref{NLS mild solution 2} and \eqref{NLS mild solution 3} we deduce that 
\begin{equation}
\label{Case 2 lower bound on v}
\|\phi^{(1)}(t)\|_{H^{3 \varrho-d}} \geq \bigg(\frac{\|\phi_0\|_{L^2_x}}{R^{1-\theta}}\bigg)^{1/\theta}\,.
\end{equation}
Combining \eqref{Gronwall 1}, \eqref{Case 1 lower bound on v}, and \eqref{Case 2 lower bound on v}, we deduce that
\begin{equation*}
|a(t)-1| \lesssim \int_0^{t} |a(\tau)-1|\,d\tau\,,
\end{equation*}
from where we conclude that $a = 1$. The result now follows.
\end{proof}
The remainder of this section is devoted to the proof of Proposition \ref{factorized_GP_uniqueness}. 
The proof of Proposition \ref{Product estimate proposition} is given in Appendix \ref{Appendix A} below. We first note several auxiliary facts, which are based on the arguments of \cite{ChHaPavSei}.
Given $t \in [-T,T]$ we let 
\begin{equation}
\label{dmu}
\dd \mu^{(1)}_t \deq \delta_{\phi^{(1)}(\,,t)}\,, \quad \dd 
\mu^{(2)}_t \deq \delta_ {\phi^{(2)}(\,,t)}\,.
\end{equation}
Note that these are Borel measures on $L^2(\mathbb{D})$.
Following \cite{ChHaPavSei}, we define for $k,r \in \mathbb{N}$ the set
\begin{multline}
\label{M_kr}
\mathcal{M}_{k,r} \deq \Big\{\sigma: \{k+1, k+2, \ldots, k+r\} \rightarrow \{1, 2, \ldots, k+r-1\}:\,
\\\sigma(j)<j \,\, \mbox{for all }j \in \{k+1, k+2, \ldots, k+r\}  \Big\}\,,
\end{multline}
and the quantity
\begin{align}
\notag
&J^k(f;\sigma;t,t_1,\ldots,t_r; \vec{x}_k;\vec{x}_k') \deq 
\mathcal{U}^{(k)} (t-t_1) \, B_{\sigma(k+1), \,k+1} \, \mathcal{U}^{(k+1)}(t_1-t_2) \, B_{\sigma(k+2),\,k+2}
\\
\label{J^k}
&\cdots \, \mathcal{U}^{(k+r-1)} (t_{r-1}-t_r) \, B_{\sigma(k+r),\,k+r} \big( |f \rangle \langle f|^{\otimes (k+r)}\big) (\vec{x}_k;\vec{x}'_k)\,.
\end{align}

\begin{lemma} 
\label{sigma_bound}
Let $s>s_0$. There exists a constant $C>0$ depending only on the domain $\mathbb{D}$ and $s$ such that for $\sigma \in \mathcal{M}_{k,r}, t \in [-T,T]$ and $q \in \{1,2\}$ we have
\begin{align}
\notag
&\int_{[0,t]^r} \int_{L^2(\mathbb{D})} \dd \mu_{t_r}^{(q)}(f) \tr \big|S^{(k,-\zeta_0)} J^k(f;\sigma;t,t_1,\ldots,t_r;\vec{x}_k;\vec{x}_k')\big| \,\dd t_r \, \cdots \, \dd t_1 
\\
\label{sigma_bound_claim}
&\leq C^r t^{\epsilon(r-1)} \, \int_{0}^{t} \|\phi^{(q)}(\,,t_r)\|_{H^s}^{2(k+r)}\,\dd t_r\,.
\end{align}
Here $\epsilon \equiv \epsilon(\mathbb{D})$ is the parameter from Definition \ref{Admissible_domain}.
\end{lemma}

\begin{proof}[Proof of Lemma \ref{sigma_bound}]
We can rewrite \eqref{J^k} as a product of $1$-particle kernels as in \cite[(4.26)]{ChHaPavSei}
\begin{equation}
\label{J^k_1}
J^k(f;\sigma;t,t_1,\ldots,t_r;\vec{x}_k;\vec{x}_k') = \prod_{j=1}^{k} J^1_j (f;\sigma_j; t, t_{\ell_{j,1}}, \ldots, t_{\ell_{j,m_j}};x_j;x_j')\,.
\end{equation}
For the precise definitions and explanations of the notation, we refer the reader to the discussion in \cite[Section 5]{ChHaPavSei}.
Following the terminology of \cite{ChHaPavSei}, one of the factors on the right-hand side of \eqref{J^k_1} is \emph{distinguished}, whereas the other $k-1$ are \emph{regular}. The distinguished term contains a factor of $|f|^2f$.

Let us first explain the modifications needed to analyze the contribution from the distinguished factor.
We change the expression in \cite[eq. (6.4)]{ChHaPavSei} to \footnote{We change the convention from \cite{ChHaPavSei} and denote elements of $L^2(\mathbb{D})$ by $f$ instead of $\phi$.}
\begin{equation*}
\int_{[0,t]^{m_j-1}} \tr \big|S^{(1,-\zeta_0)}J_j^1(f;\sigma_j;t,t_{\ell_{j,1}}, \ldots, t_{\ell_{j,m_j}})\big| \, \dd t_{\ell_{j,m_j-1}} \, \dd t_{\ell_{j,m_j-2}} \cdots \, \dd t_{\ell_{j,1}}\,.
\end{equation*}

Therefore, the expression in \cite[eq. (7.1)]{ChHaPavSei} gets changed to
\begin{equation*}
\sum_{\beta_1} \mathop{\int}_{[0,t]^{m_j-1}} \|\psi_{\beta_1}^{1}\|_{H^{-\zeta_0}} \, \|\chi_{\beta_1}^{1}\|_{H^{-\zeta_0}} \, \dd t_{\ell_{j,m_j-1}} \,\dd t_{\ell_{j,m_j-2}} \cdots \, \dd t_{\ell_{j,1}}\,
\end{equation*}
which is
\begin{equation}
\notag
\leq \sum_{\beta_1} \mathop{\int}_{[0,t]^{m_j-1}} \|\psi_{\beta_1}^1\|_{B^{-\zeta_0}_{2,1}} \, \|\chi_{\beta_1}^1\|_{B^{\zeta}_{2,1}} \,dt_{\ell_{j,m_j-1}} \,dt_{\ell_{j,m_j-2}} \,\cdots \,dt_{\ell_{j,1}}
\end{equation}
for some $\zeta \in (\zeta_0,s)$, which we take to be fixed for the remainder of the proof. Note that, in the last step we used \eqref{Inclusion}.
 
The key point is that we can now also change what were originally the \emph{bound on the $L^2$ level} \cite[eq. (7.4)]{ChHaPavSei} and the \emph{bound on the $H^1$ level} \cite[eq. (7.5)]{ChHaPavSei} by using the estimates that hold by assumption from Definition \ref{Admissible_domain}. In particular, we can use \eqref{Star1} and \eqref{Star2} from Definition \ref{Admissible_domain} and replace \cite[eq. (7.4)]{ChHaPavSei}  with

\begin{align}
\notag
&\int_{[0,t]^{d_a}} \|\psi_{\beta_a}^{a}\|_{B^{-\zeta_0}_{2,1}} \, \|\chi_{\beta_a}^{a}\|_{B^{\zeta}_{2,1}} \prod_{a' \in \tau_{j,a}} \,\dd t_{\ell_{j,a'}}
\\
\notag
&\lesssim t^{\,\epsilon} \, \bigg( \int_{[0,t]^{d_{\kappa_{-}(a)}}} \|\psi_{\beta_{\kappa_{-}(a)}}^{\kappa_{-}(a)}\|_{B^{\zeta}_{2,1}} \, \|\chi_{\beta_{\kappa_{-}(a)}}^{\kappa_{-}(a)}\|_{B^{\zeta}_{2,1}} \prod_{a' \in \tau_{j,\kappa_{-}(a)}} \, \dd t_{\ell_{j,a'}} \bigg)
\\
\label{Bound_on_the_low_regularity_level}
&\times \bigg( \int_{[0,t]^{d_{\kappa_{+}(a)}}} \|\psi_{\beta_{\kappa_{+}(a)}}^{\kappa_{+}(a)}\|_{B^{-\zeta_0}_{2,1}} \, \|\chi_{\beta_{\kappa_{+}(a)}}^{\kappa_{+}(a)}\|_{B^{\zeta}_{2,1}} \prod_{a' \in \tau_{j,\kappa_{+}(a)}} \, \dd t_{\ell_{j,a'}} \bigg)\,.
\end{align}
We always apply this estimate in the case where $\psi_{\beta_{\kappa_{+}(a)}}^{\kappa_{+}(a)}$ contains the factor $|f|^2f$.
Furthermore, we can use \eqref{Star2} and replace \cite[eq. (7.5)]{ChHaPavSei} with
\begin{align}
\notag
&\int_{[0,t]^{d_a}} \|\psi_{\beta_a}^{a}\|_{B^{\zeta}_{2,1}} \, \|\chi_{\beta_a}^{a}\|_{B^{\zeta}_{2,1}} \prod_{a' \in \tau_{j,a}} \,\dd t_{\ell_{j,a'}}
\\
\notag
&\lesssim t^{\,\epsilon} \, \bigg( \int_{[-T,T]^{d_{\kappa_{-}(a)}}} \|\psi_{\beta_{\kappa_{-}(a)}}^{\kappa_{-}(a)}\|_{B^{\zeta}_{2,1}} \, \|\chi_{\beta_{\kappa_{-}(a)}}^{\kappa_{-}(a)}\|_{B^{\zeta}_{2,1}} \prod_{a' \in \tau_{j,\kappa_{-}(a)}} \, \dd t_{{\ell_{j,a'}}} \bigg)
\\
\label{Bound_on_the_high_regularity_level}
&\times \bigg( \int_{[0,t]^{d_{\kappa_{+}(a)}}} \|\psi_{\beta_{\kappa_{+}(a)}}^{\kappa_{+}(a)}\|_{B^{\zeta}_{2,1}} \, \|\chi_{\beta_{\kappa_{+}(a)}}^{\kappa_{+}(a)}\|_{B^{\zeta}_{2,1}} \prod_{a' \in \tau_{j,\kappa_{+}(a)}} \, \dd t_{{\ell_{j,a'}}} \bigg)\,.
\end{align}
Finally, we note that we replace applications of \cite[eq. (7.2)]{ChHaPavSei}), i.e.\
\begin{equation}
\notag
\||f|^2 f\|_{L^2} \lesssim \|f\|_{H^1}^3
\end{equation}
by \eqref{Star3} of Definition \ref{Admissible_domain}, which implies that 
\begin{equation}
\label{Estimate on the nonlinearity application}
\||f|^2 f\|_{B^{-\zeta_0}_{2,1}} \lesssim \|f\|_{H^s}^3\,. 
\end{equation}
Using \eqref{Bound_on_the_low_regularity_level}, \eqref{Bound_on_the_high_regularity_level}, and \eqref{Estimate on the nonlinearity application}, it follows that the bound \cite[eq. (8.1)]{ChHaPavSei} for the distinguished tree is replaced with

\begin{align}
&\mathop{\int}_{[0,t]^{m_j-1}} \tr \big|S^{(1,-\zeta_0)}J_j^1(f;\sigma_j;t,t_{\ell_{j,1}},\ldots,t_{\ell_{j,m_j}})\big| \, dt_{\ell_{j,m_j-1}} \cdots \, dt_{\ell_{j,1}}
\\
\label{Distinguished tree bound}
&\leq C_1^{m_j} \, t^{(m_j-1) \epsilon} \, \|f\|_{B^{\zeta}_{2,1}}^{2m_j-1} \, \|f\|_{H^s}^3 \leq C^{m_j} \, t^{(m_j-1) \epsilon} \, \|f\|_{H^s}^{2m_j+2}\,,
\end{align}
where $C_1,C>0$ depend on $\mathbb{D}$ and $s$. In the last step we used \eqref{Inclusion}.

The contribution from each regular tree, i.e.\  the bound \cite[eq. (8.6)]{ChHaPavSei} is replaced with

\begin{align}
\notag
\mathop{\int}_{[0,t]^{m_j}}  \tr \big|S^{(1,-\zeta_0)} J_j^1(f;\sigma_j;t,t_{\ell_{j,1}},\ldots,t_{\ell_{j,m_j}})\big| \, dt_{\ell_{j,m_j}} \,dt_{\ell_{j,m_j-1}} \cdots \,dt_{\ell_{j,1}}
\\
\label{Regular tree bound}
\leq C_2^{m_j} \, t^{m_j  \epsilon} \, \|f\|_{B^{\zeta}_{2,1}}^{2m_j+2}
\leq C^{m_j} \, t^{m_j  \epsilon} \, \|f\|_{H^s}^{2m_j+2}\,,
\end{align}
where $C_2,C>0$ depend on $\mathbb{D}$ and $s$.
In order to obtain \eqref{Regular tree bound}, one applies analogous modifications to those needed to obtain \eqref{Distinguished tree bound}. The difference is that one now only uses \eqref{Star2} and does not need to use \eqref{Star1} and \eqref{Star3}.

Using \eqref{Distinguished tree bound} and \eqref{Regular tree bound} in \eqref{J^k_1}, it follows that the left-hand side of \eqref{sigma_bound_claim} is 
\begin{equation*}
\leq C^r t^{\epsilon(r-1)} \, \int_{0}^{t} \int_{L^2(\mathbb{D})} \dd \mu_t^{(q)}(f) \|f\|_{H^s}^{2(k+r)}\,\dd t_r
 =C^r t^{\epsilon(r-1)} \, \int_{0}^{t} \|\phi^{(q)}(\,,t_r)\|_{H^s}^{2(k+r)}\,\dd t_r\,.
\end{equation*}
\end{proof}

The \emph{boardgame argument} of \cite[Section 3]{KM} gives an explicit equivalence relation $\sim$ on the set 
$\mathcal{M}_{k,r}$ defined in \eqref{M_kr} above.
Furthermore, there exists a complete set of representatives $\mathcal{N}_{k,r}$ with the property that
\begin{equation}
\label{N_kr_bound}
\sharp \,\mathcal{N}_{k,r} \leq 2^{k+r}\,.
\end{equation}
We omit the details of the construction and refer the reader to \cite[Section 3]{KM}.

\begin{lemma} 
\label{Boardgame_argument}For $\sigma_0 \in \mathcal{N}_{k,r}$, there exists $D(\sigma_0,t) \subset [0,t]^r$ such that for $q=1,2$, we have
\begin{multline}
\label{boardgame_argument_claim}
\mathop{\sum_{\sigma \in \mathcal{M}_{k,r}}}_{\sigma \sim \sigma_0} \int_0^t \int_0^{t_1} \cdots \int_0^{t_{r-1}}  \int_{L^2(\mathbb{D})}\, \dd \mu_{t_r}^{(q)}(f) J^k(f;\sigma;t,t_1,\ldots,t_r)\,\, \dd t_r \, \cdots \, \dd t_1
\\
= \int_{D(\sigma_0,t)} \int_{L^2(\mathbb{D})}\, \dd \mu_{t_r}^{(q)}(f)  J^k(f;\sigma_0;t,t_1,\ldots,t_r) \,\, \dd t_r \, \cdots \, \dd t_1.
\end{multline}
\end{lemma}

\begin{proof}[Proof of Lemma \ref{Boardgame_argument}]
This claim follows from the analysis of \cite[Section 3]{KM}. A subtle point that one should address is that this proof relies on the commutation relation \cite[eq. (37)]{KM}, which in turn relies on considering operator kernels. To fully justify the application of this commutation relation, we work with functions which are smooth in the spatial variables and conclude \eqref{boardgame_argument_claim} by density. More precisely, we find a sequence $F_m$ of smooth functions on $[0,t]  \times \mathbb{D}$ which converge to $\phi^{(q)}$ in $L^{2(k+r)}_{[0,t]} H^s_x$. Here we think of $\phi^{(q)}$ as a function of space and time. For $\tau \in [0,t]$, 
we let $\dd \nu_\tau^{(m)} \deq \delta_{F_m(\,,\tau)}$. Then, since for all $\tau \in [0,t]$, the function $F(\,, \tau)$ is smooth in $x$, the proof of \cite[Lemma 3.1]{KM} carries over and allows us to deduce the analogue of \eqref{boardgame_argument_claim} with $\dd \mu^{(q)}$ replaced by $\dd \nu^{(m)}$. The claim now follows by Lemma \ref{sigma_bound} and letting $m \rightarrow \infty$.
\end{proof}
We can now prove Proposition \ref{factorized_GP_uniqueness}.
\begin{proof}[Proof of Proposition \ref{factorized_GP_uniqueness}] 
Given $k \in \mathbb{N}$ we define
\begin{equation*}
\gamma^{(k)}(\vec{x}_k;\vec{x}'_k,t):=\big|\phi^{(1)} \rangle \langle \phi^{(1)}\big|^{\otimes k} (\vec{x}_k;\vec{x}'_k,t)\,,\widetilde{\gamma}^{(k)}(\vec{x}_k;\vec{x}'_k,t):=\big|\phi^{(2)} \rangle \langle \phi^{(2)}\big|^{\otimes k} (\vec{x}_k;\vec{x}'_k,t)\,.
\end{equation*}
By Lemma \ref{Mild factorized solutions}, we know that $(\gamma^{(k)})_k,(\widetilde{\gamma}^{(k)})_k$ are mild solutions of \eqref{GP hierarchy} in $\mathfrak{H}^{s}$ with the same initial data.
In particular
\begin{equation*}
\widehat{\gamma}^{(k)} \deq \gamma^{(k)}-\widetilde{\gamma}^{(k)}\,,
\end{equation*}
is a mild solution of \eqref{GP hierarchy} in $\mathfrak{H}^{s}$ with zero initial data.  
It suffices to show that
\begin{equation}
\label{gamma_hat_claim}
\tr \big| S^{(k,-\zeta_0)} \widehat{\gamma}^{(k)} \big| = 0 \,\quad \mbox{on } [-T,T]\,,
\end{equation}
for all $k \in \mathbb{N}$ provided that we choose $T$ to be sufficiently small depending on $R$ (from Definition \ref{Mild Solution of the NLS}). The claim for general $T$ follows by an iteration argument.

Recalling \eqref{dmu}, we define
\begin{equation}
\label{dmu_hat}
\dd \widehat{\mu}_t \deq \dd \mu^{(1)}_t-\dd \mu^{(2)}_t\,.
\end{equation}
By applying an iterated Duhamel expansion, we have that for all $k,r \in \mathbb{N}$ and $t \in [-T,T]$
\begin{multline}
\widehat{\gamma}^{(k)}(t) = (-\ii)^r\, \sum_{\sigma \in \mathcal{M}_{k,r}} \int_0^t \int_{0}^{t_1} \cdots \int_{0}^{t_{r-1}} \int_{L^2(\mathbb{D})} 
\dd \widehat{\mu}_{t_r} (f) 
\\
\label{Duhamel_expansion}
J^k(f;\sigma;t,t_1,\ldots,t_r; \vec{x}_k;\vec{x}_k')\, \dd t_r \, \cdots \, \dd t_1 \,,
\end{multline}
for $\mathcal{M}_{k,r}$ as in \eqref{M_kr} and $J^k$ as in \eqref{J^k}.
Using Lemma \ref{Boardgame_argument} and \eqref{dmu_hat} we can rewrite the expression on the right hand side of \eqref{Duhamel_expansion} as
\begin{equation}
\label{Duhamel_expansion2}
 (-\ii)^r \, \sum_{\sigma_0 \in \mathcal{N}_{k,r}} \int_{D(\sigma_0,t)} \int_{L^2(\mathbb{D})}\, \dd \widehat{\mu}_{t_r}(f)  J^k(f;\sigma_0;t,t_1,\ldots,t_r) \,\, \dd t_r \, \cdots \, \dd t_1\,.
\end{equation}
Using the triangle inequality, \eqref{N_kr_bound}, Lemma \ref{sigma_bound}, and the assumptions on $\phi^{(1)},\phi^{(2)}$ in the expression \eqref{Duhamel_expansion2}, it follows that for $t \in [-T,T]$ we have
\begin{equation*}
\tr \big| S^{(k,-\zeta_0)} \widehat{\gamma}^{(k)} \big| \leq 2^{k+r+1} C^r t^{\epsilon(r-1) + 1 } R^{2(k+r)} = 2^{k+1} t^{1-\epsilon} R^{2k} \big(2CR^2 t^{\epsilon} \big)^r\,.
\end{equation*}
We deduce \eqref{gamma_hat_claim} by letting $r \rightarrow \infty$.
\end{proof}

\begin{remark}
\label{Previous_paper}
Note that the proof of Proposition \ref{factorized_GP_uniqueness} implies that mild solutions of \eqref{GP hierarchy} in $\mathfrak{H}^s$ of the form 
\begin{equation*}
\gamma^{(k)}(t) = \int_{L^2(\mathbb{D})} \, \dd \mu_t (f) \, |f \rangle \langle f |^{\otimes k}
\end{equation*}
for $\dd \mu_t$ a Borel measure on $L^2(\mathbb{D})$ supported on a ball centered at zero of radius independent of $t$, are uniquely determined by their initial data provided that $s>s_0$. In light of Proposition \ref{Torus_admissibility} below, this improves on our earlier uniqueness results for such solutions on rectangular tori \cite{HerrSohinger}.
\end{remark}
\begin{remark}
Note that, in the proof of Theorem \ref{Main_theorem}, we never directly use the quantum de Finetti theorem. For our purposes, it suffices to apply the conclusion of this theorem concerning the structure of density matrices. This, in turn, allows us to apply the techniques developed in \cite{ChHaPavSei}. In doing so we work in $L^2$-based Besov spaces, which allows us to avoid the issues that arise when working in $L^p$-based spaces. The latter approach was taken on $\mathbb{R}^d$ in \cite{HTX,HTX2} and was based on the stronger dispersive properties available on the whole space.
\end{remark}
\section{Admissibility of the rectangular torus of dimension $d \geq 2$}
\label{torus_admissible}

In this section we consider the general rectangular torus in $d \geq 2$ dimensions, which we write as

\begin{equation*}
\mathbb{T}^d_g=\Big(\mathbb{R} \big/\,\frac{2\pi}{\theta_1} \mathbb{Z}\Big) \times
\Big(\mathbb{R} \big/\,\frac{2\pi}{\theta_2} \mathbb{Z}\Big) \times \cdots \times \Big(\mathbb{R} \big/\,\frac{2\pi}{\theta_d} \mathbb{Z}\Big)
\end{equation*}
for some $\theta_1,\theta_2,\ldots,\theta_d>0$ and we show that it is admissible in the sense of Definition \ref{Admissible_domain} above.
\begin{proposition}
\label{Torus_admissibility}
$\mathbb{D}=\mathbb{T}^d_g$ is admissible with the following parameters:

\begin{equation}
\label{Definition of sigma}
\zeta_0=
\begin{cases}
\frac{d  (d-1)}{2(d+2)}\,\,\, &\mbox{for } 2 \leq d \leq 4 \\
\frac{d}{2}-1\,\,\, &\mbox{for }d \geq 5\,.
\end{cases}
\end{equation}
\begin{equation}
\label{Definition of alpha}
\alpha_0=
\begin{cases}
\frac{d  (d+5)}{6(d+2)}\,\,\, &\mbox{for }2 \leq d \leq 4\\
\frac{d}{6}+\frac{1}{3}\,\,\, &\mbox{for }d \geq 5\,.
\end{cases}
\end{equation}
\begin{equation}
\label{Definition of epsilon(d)}
\epsilon=
\begin{cases}
\Big(\frac{4-d}{2(d+2)}\Big)+\,\,\, &\mbox{for } 2 \leq d \leq 4\\
\quad 0+\,\,\, &\mbox{for } d \geq 5\,.
\end{cases}
\end{equation}

\end{proposition}

From Theorem \ref{Main_theorem} and Proposition \ref{Torus_admissibility} we obtain 
Corollary \ref{unconditional_uniqueness_torus} stated earlier.
In the proof of Proposition \ref{Torus_admissibility} we use the Strichartz estimate on $\mathbb{T}^d_g$ \cite{Bourgain_Demeter1,Killip_Visan}.
\begin{proposition}[Strichartz estimate on $\mathbb{T}^d_g$]
\label{Strichartz Estimate}
For $p>\frac{2(d+2)}{d}$ we have
\begin{equation}
\notag
\|e^{it\Delta}P_{\leq N}f\|_{L^p_{t,x}([0,1] \times \mathbb{T}^d_g)} \lesssim_{d,p} \langle N \rangle^{\frac{d}{2}-\frac{d+2}{p}} \,\|f\|_{L^2_x}\,.
\end{equation}
Here 
\begin{equation*}
P_{\leq N}:=\sum_{N' \leq N} P_{N'}\,.
\end{equation*}
\end{proposition}
Moreover we recall that the Bernstein inequality holds for the full range of integrability exponents, as on $\mathbb{R}^n$. We note that for general domains $\mathbb{D}$, we only have the claim for a partial range of integrability exponents, see Lemma \ref{Bernstein1_lemma} below.

Let $\Z^d_g=(\theta_1\Z\times \cdots \times \theta_d\Z)$ and $\widetilde{P}_N$ be a smoothed out version of $P_N$, i.e.\
\begin{equation}
\label{P_N_smooth}
\widetilde{P}_N \deq \sum_{k\in \Z^d_g} \varphi_N(k) \chi_k\,,
\end{equation}
where, if $N\geq 1$, $\varphi_N=\varphi(\cdot/N)$ and $\varphi \in C_c^\infty(\mathbb{R}^d)$ equal to one on $1 \leq |\xi| <2$ and supported on $1/2 \leq |\xi| \leq 4$, and $\varphi_0 \in C_c^\infty(\mathbb{R}^d)$ equal to one on $|\xi| <1$ and supported on $|\xi| \leq 2$. In particular, we have $\widetilde{P}_N P_N=P_N$.
The following is well-known, for the reader's convenience we provide a short proof
in Appendix \ref{appendix-bern}.
\begin{lemma}[Bernstein's inequality on the rectangular torus]
\label{Bernstein_inequality_torus}
Let $1 \leq p \leq q \leq \infty$. For all $N$ we have
\begin{equation}
\label{Bernstein_inequality_torus_bound}
\|\widetilde{P}_N f\|_{L^q} \lesssim \langle N \rangle^{d/p-d/q}  \|f\|_{L^p}.
\end{equation}
\end{lemma}
We now give the proof of Proposition \ref{Torus_admissibility}
\begin{proof}[Proof of Proposition \ref{Torus_admissibility}]
We first prove \eqref{Star1} for the choice of parameters as in \eqref{Definition of sigma}, \eqref{Definition of alpha}, and \eqref{Definition of epsilon(d)} above. In the sequel we abbreviate $\|\cdot\|_{L^p([-T,T])}$ as $\|\cdot\|_{L^p_T}$.
Let us note for $0 \leq \eta \leq \zeta_0$ we have
\begin{equation}
\label{Big Sum}
\|u_1 u_2 u_3 \|_{L^1_t B^{-\eta}_{2,1}([-T,T] \times \mathbb{T}^d_g)} \leq \sum_{N} \langle N \rangle^{-\eta} \, \|P_N(u_1 u_2 u_3)\|_{L^1_T L^2_x} =: I\,.
\end{equation}
We write
\begin{equation*}
P_N(u_1 u_2 u_3)=\sum_{N_1,N_2,N_3} P_N\big(u_{1,N_1} \, u_{2,N_2} u_{3,N_3}\big)\,,
\end{equation*}
where we abbreviate $u_{j,N_j}:=P_{N_j}u_j$ for $j=1,2,3$. We remark that there is no contribution unless $N \lesssim \max\,\{N_1,N_2,N_3\}$. We now consider two cases, depending on the relative size of $\max\,\{N_1,N_2,N_3\}$ with respect to $N$. 

\bigskip

\textbf{Case 1:} $\max\,\{N_1,N_2,N_3\} \lesssim N$.

\medskip

\textbf{Case 1.A:} $N_1=\max\,\{N_1,N_2,N_3\}$.
In this case there is no contribution unless $N\sim N_1$ and we obtain
\begin{equation}
\label{sum Case 1A}
I \lesssim \mathop{\sum_{N_1,N_2,N_3}}_{N_1 \geq N_2,N_3} \langle N_1 \rangle^{-\eta} \, \|u_{1,N_1} \,u_{2,N_2} \,u_{3,N_3}\|_{L^1_T L^2_x}\,.
\end{equation}
With the abbreviation $\phi_{j,N_j}:=P_{N_j}\phi_j$ for $j=1,2,3$, it suffices to show that
\begin{align}
\notag
&\|u_{1,N_1} \,u_{2,N_2} \,u_{3,N_3}\|_{L^1_T L^2_x}
\\
\label{star}
&\lesssim
T^{\epsilon} \,\|\phi_{1,N_1}\|_{L^2_x} \,\Big(\langle N_2 \rangle^{\zeta_0+} \,\|\phi_{2,N_2}\|_{L^2_x} \Big) \,\Big(\langle N_3 \rangle^{\zeta_0+} \,\|\phi_{3,N_3}\|_{L^2_x}\Big)\,.
\end{align}

In order to establish \eqref{star} we consider two subcases.

\medskip

\textbf{Subcase 1.A.i:} $N_1 \lesssim \max \,\{N_2,N_3\}$.
By H\"{o}lder's inequality we have
\begin{align}
\notag
&\|u_{1,N_1} \,u_{2,N_2} \,u_{3,N_3}\|_{L^1_T L^2_x} 
\\
\label{Subcase 1A1 Holder's inequality}
&\leq \|u_{1,N_1}\|_{L^{\frac{2(d+2)}{d}}_{T,x}} \,\|u_{2,N_2}\|_{L^{\frac{4(d+2)}{d+4}}_T L^{2(d+2)}_x} \,\|u_{3,N_3}\|_{L^{\frac{4(d+2)}{d+4}}_T L^{2(d+2)}_x}\,.
\end{align}

We will apply \eqref{Subcase 1A1 Holder's inequality} differently depending on whether $d \leq 4$ or $d \geq 5$.

\medskip

\textbf{(1) $2 \leq d \leq 4$.}
In this case, we have 
$
\tfrac{4(d+2)}{d+4} \leq \tfrac{2(d+2)}{d} \leq 2(d+2).
$
We apply Lemma \ref{Bernstein_inequality_torus} (which is possible since $u_{2,N_2}=\widetilde{P}_{N_2} u_{2,N_2}$), H\"{o}lder's inequality, and Proposition \ref{Strichartz Estimate} with $p=\frac{2(d+2)}{d}+$ to deduce that
\begin{align*}
\notag
\|u_{2,N_2}\|_{L^{\frac{4(d+2)}{d+4}}_T L^{2(d+2)}_x} 
\lesssim T^{\,\frac{4-d}{4(d+2)}+} \,\langle N_2 \rangle^{\,\frac{d \,(d-1)}{2(d+2)}+} \, \|\phi_{2,N_2}\|_{L^2_x}\,.
\end{align*}
An analogous bound holds for $u_{3,N_3}$. Using these bounds and Proposition \ref{Strichartz Estimate} with $p=\frac{2(d+2)}{d}+$ for the $u_{1,N_1}$ factor, we get that the expression \eqref{Subcase 1A1 Holder's inequality} is 
\begin{align*}
\notag
&\lesssim T^{\,\frac{4-d}{2(d+2)}+} \,\Big(\langle N_1 \rangle^{\,0+} \,\|\phi_{1,N_1}\|_{L^2_x} \Big) \,\Big(\langle N_2 \rangle^{\,\frac{d \,(d-1)}{2(d+2)}+} \,\|\phi_{2,N_2}\|_{L^2_x}\Big) \,\Big(\langle N_3 \rangle^{\,\frac{d \,(d-1)}{2(d+2)}+} \,\|\phi_{3,N_3}\|_{L^2_x}\Big)
\\
&\lesssim  T^{\,\frac{4-d}{2(d+2)}+} \,\|\phi_{1,N_1}\|_{L^2_x} \,\Big(\langle N_2 \rangle^{\,\frac{d \,(d-1)}{2(d+2)}+} \,\|\phi_{2,N_2}\|_{L^2_x}\Big) \,\Big(\langle N_3 \rangle^{\,\frac{d \,(d-1)}{2(d+2)}+} \,\|\phi_{3,N_3}\|_{L^2_x}\Big)\,.
\end{align*}
In the last line, we used the fact that $N_1 \lesssim \max \,\{N_2,N_3\}$ to distribute the factor of $\langle N_1 \rangle^{0+}$.

\medskip

\textbf{(2) $d \geq 5$.}
In this case, we have
$\tfrac{2(d+2)}{d}<\tfrac{4(d+2)}{d+4}<2(d+2).
$
We first apply Lemma \ref{Bernstein_inequality_torus} to obtain that \eqref{Subcase 1A1 Holder's inequality} is
\begin{equation}
\label{d>=5_estimate}
\leq \|u_{1,N_1}\|_{L^{\frac{2(d+2)}{d}}_{T,x}} \,\Big(\langle N_2 \rangle^{\,\frac{d}{4}} \,\|u_{2,N_2}\|_{L^{\frac{4(d+2)}{d+4}}_{T,x}}\Big) \,\Big(\langle N_3 \rangle^{\,\frac{d}{4}} \,\|u_{3,N_3}\|_{L^{\frac{4(d+2)}{d+4}}_{T,x}}\Big)\,.
\end{equation} 
To estimate the $u_{2,N_2}$ and $u_{3,N_3}$ factors in \eqref{d>=5_estimate} apply Proposition \ref{Strichartz Estimate} with $p=\frac{4(d+2)}{d+4}$. The $u_{1,N_1}$ factor we estimate as we did for $d \leq 4$. Therefore, \eqref{d>=5_estimate} is 
\begin{align*}
&\lesssim T^{\,0+} \, \|\phi_{1,N_1}\|_{L^2_x} \, \Big(\langle N_2 \rangle^{\,(\frac{d}{2}-1)+} \, \|\phi_{2,N_2}\|_{L^2_x} \Big) \, \Big(\langle N_3 \rangle^{\,(\frac{d}{2}-1)+} \, \|\phi_{3,N_3}\|_{L^2_x}\Big)\,.
\end{align*}
This finishes the proof of \eqref{star} in Subcase 1.A.i.

\medskip

\textbf{Subcase 1.A.ii:} $N_1 \gg N_2,N_3$. By symmetry, we can assume without loss of generality that $N_2 \geq N_3$.
In this subcase, we can no longer distribute the factor of $N_1^{\,0+}$ into the factors involving $N_2$ and $N_3$. Instead, we need to apply an \emph{almost orthogonality} argument. More precisely, let $\Gamma$ denote the collection of all non-overlapping cubes of size $N_2$ with centers in the lattice $N_2 \mathbb{Z}^d_{g}$.
For $\mathcal{C} \in \Gamma$, we denote by $P_{\mathcal{C}}$ the projection to $\mathcal{C}$ in frequency space. Given $\mathcal{C} \in \Gamma$, we observe that
\begin{equation}
\label{almost_orthogonality}
\int_{\mathbb{T}^d_{g}} \big(P_{\mathcal{C}} \,u_{1,N_1}\big) \, u_{2,N_2} \, u_{3,N_3} \, \overline{\big(P_{\mathcal{C'}} \,u_{1,N_1}\big) \, u_{2,N_2} \, u_{3,N_3}} \,dx \neq 0
\end{equation}
only for finitely many $\mathcal{C'} \in \Gamma$, with a bound on this number independent of $\mathcal{C},N_1,N_2,N_3$.
By Galilean invariance of $e^{it \Delta}$ and by Proposition \ref{Strichartz Estimate} for $p=\frac{2(d+2)}{d}+$, it follows that

\begin{equation} 
\label{Strichartz Estimate Galilean Invariance}
\|P_{\mathcal{C}}\,u_{1,N_1}\|_{L^{\frac{2(d+2)}{d}+}_{T,x}} \lesssim \langle N_2 \rangle^{\,0+} \, \|P_{\mathcal{C}}\,\phi_{1,N_1}\|_{L^2_x}.
\end{equation}
By applying \eqref{Strichartz Estimate Galilean Invariance} instead of
Proposition \ref{Strichartz Estimate} when $p=\frac{2(d+2)}{d}+$ to
estimate the factor with frequency of order $N_1$, the arguments from
Subcase 1.A.i imply that
\begin{align}
\notag
&\big\|\big(P_{\mathcal{C}} \,u_{1,N_1}\big) \, u_{2,N_2} \, u_{3,N_3}\big\|_{L^1_T L^2_x} 
\\
\label{almost_orthogonality2}
&\lesssim T^{\,\epsilon} \, \|P_{\mathcal{C}}\,\phi_{1,N_1}\|_{L^2_x} \, \Big(\langle N_2 \rangle^{\zeta_0+} \, \|\phi_{2,N_2}\|_{L^2_x} \Big) \, \Big(\langle N_3 \rangle^{\zeta_0+} \, \|\phi_{3,N_3}\|_{L^2_x}\Big)\,.
\end{align}
Taking an $\ell^2$ sum over $\mathcal{C} \in \Gamma$ in
\eqref{almost_orthogonality2} and using \eqref{almost_orthogonality},
we deduce \eqref{star} in Subcase 1.A.ii. In summary, in Case 1.A we obtain

\begin{align}
\notag
I \lesssim{} &\mathop{\sum_{N_1,N_2,N_3}}_{N_1 \geq N_2,N_3} \langle N_1 \rangle^{-\eta} \, \|u_{1,N_1} \, u_{2,N_2} \, u_{3,N_3}\|_{L^1_T L^2_x} 
\\
\notag
\lesssim{}& \mathop{\sum_{N_1,N_2,N_3}}_{N_1 \geq N_2,N_3} T^{\,\epsilon} \, \Big(\langle N_1 \rangle^{-\eta} \, \|\phi_{1,N_1}\|_{L^2_x} \Big) \, \Big(\langle N_2 \rangle^{\zeta_0+} \, \|\phi_{2,N_2}\|_{L^2_x} \Big) \, \Big(\langle N_3 \rangle^{\zeta_0+} \, \|\phi_{3,N_3}\|_{L^2_x}\Big)
\\
\label{Case 1A bound}
\lesssim{}& T^{\,\epsilon} \, \|\phi_1\|_{B^{-\eta}_{2,1}} \, \|\phi_2\|_{B^{\zeta}_{2,1}} \, \|\phi_3\|_{B^{\zeta}_{2,1}}\,.
\end{align}

\medskip

\textbf{Case 1.B:} $N_2=\max\,\{N_1,N_2,N_3\}$. 
Since $N \sim N_2 \geq N_1$, it suffices to estimate
\begin{equation}
\label{sum Case 1B}
\mathop{\sum_{N_1,N_2,N_3}}_{N_2 \geq N_1,N_3} \langle N_1 \rangle^{-\eta} \, \|u_{1,N_1} \, u_{2,N_2} \, u_{3,N_3}\|_{L^1_T L^2_x}.
\end{equation}
We estimate each summand in \eqref{sum Case 1B} by using \eqref{star} as in Case 1. 
The point is that, since $N_1 \leq N_2$, a factor of $\langle N_1 \rangle^{\,0+}$ can be absorbed into the powers of $\langle N_2 \rangle$. 
Therefore, we can deduce that the contribution from Case 1.B also satisfies the bound \eqref{Case 1A bound}.

\medskip

\textbf{Case 1.C:} $N_3=\max\,\{N_1,N_2,N_3\}$. 
Case 1.C is analogous to Case 1.B if we reverse the roles of $N_2$ and $N_3$.

\medskip

\textbf{Case 2:} $\max\,\{N_1,N_2,N_3\} \gg N$.
We observe that $P_N(u_{1,N_1} u_{2,N_2} u_{3,N_3})=0$ unless the largest two elements of $\{N_1,N_2,N_3\}$ are comparable.
We consider three possible cases.

\medskip

\textbf{Case 2.A:} $N_1 \sim N_2 \gtrsim N_3$.
In this case, we know that $N_1 \sim N_2 \gg N$.
Let $q \in [1,2]$ be fixed. We have

\begin{align}
\notag
& \langle N \rangle^{-\eta} \, \|P_N(u_{1,N_1} \, u_{2,N_2} \, u_{3,N_3})\|_{L^1_T L^2_x} \leq
\langle N \rangle^{-\eta} \, \|\widetilde{P}_N(u_{1,N_1} \, u_{2,N_2} \, u_{3,N_3})\|_{L^1_T L^2_x} 
\\
\label{Case2A product bound1} 
&\lesssim \langle N \rangle^{-\delta} \, \langle N \rangle^{-\eta+\frac{d}{q}-\frac{d}{2}+\delta} \,  \|u_{1,N_1} \, u_{2,N_2} \, u_{3,N_3}\|_{L^1_T L^q_x}\,.
\end{align}
For the last inequality, we used Lemma \ref{Bernstein_inequality_torus}.
Here $\delta>0$ is a small, arbitrary constant. We now choose $q$ to be $q_0 \equiv q_0(d)$ such that $\frac{d}{q_0}-\frac{d}{2}=\zeta_0$, for $\zeta_0$ as in \eqref{Definition of sigma}. In particular, we have
\begin{equation}
\label{choice_of_q}
q_0 = 
\begin{cases}
\frac{2(d+2)}{2d+1}\,\,\, &\mbox{for } d \leq  4\\
\frac{d}{d-1}\,\,\, &\mbox{for } d \geq 5\,.
\end{cases}
\end{equation}
Then $q_0 \in [1,2]$. Furthermore, the condition on $\eta$ in \eqref{Condition_on_eta} can then be rewritten as
\begin{equation*}
0 \leq \eta \leq \frac{d}{q_0}-\frac{d}{2}\,.
\end{equation*}
In particular, choosing $q=q_0$, the expression \eqref{Case2A product bound1} is
\begin{align}
\notag
&\lesssim \langle N \rangle^{-\delta} \, \langle N_1 \rangle^{-\eta+\zeta_0+\delta} \, \|u_{1,N_1} \, u_{2,N_2} \, u_{3,N_3}\|_{L^1_T L^q_x}
\\
\label{Case2A product bound1B}
&\sim \langle N \rangle^{-\delta} \, \langle N_1 \rangle ^{-\eta} \, \langle N_2 \rangle^{\zeta_0+\delta} \, \|u_{1,N_1} \, u_{2,N_2} \, u_{3,N_3}\|_{L^1_T L^q_x}\,.
\end{align}
Therefore, it remains to estimate
\begin{equation}
\label{Case2A product bound 2}
\langle N_1 \rangle^{-\eta} \, \langle N_2 \rangle^{\zeta_0+\delta} \, \|u_{1,N_1} \, u_{2,N_2} \, u_{3,N_3}\|_{L^1_T L^q_x}
\end{equation}
for fixed dyadic $N_1,N_2,N_3$ with $N_1 \sim N_2 \gtrsim N_3$. As in Case 1 we consider the cases $d \leq 4$ and $d \geq 5$ separately.

\medskip
\textbf{(1)} $2 \leq d \leq 4$.
By H\"{o}lder's inequality we note that, for $q$ as in \eqref{choice_of_q} we have
\begin{equation}
\label{Case 2A Holder's inequality}
\|u_{1,N_1} \, u_{2,N_2} \, u_{3,N_3}\|_{L^1_TL^q_x} \leq \|u_{1,N_1}\|_{L^{\frac{2(d+2)}{d}}_{T,x}} \, \|u_{2,N_2}\|_{L^{\frac{2(d+2)}{d}}_{T,x}} \, \|u_{3,N_3}\|_{L^{\frac{d+2}{2}}_T L^{2(d+2)}_x}\,.
\end{equation}
By Lemma \ref{Bernstein_inequality_torus}, H\"{o}lder's inequality, and Proposition \ref{Strichartz Estimate}, we can estimate the third factor in \eqref{Case 2A Holder's inequality} as
\begin{equation}
\label{3rd_factor_Case2}
\|u_{3,N_3}\|_{L^{\frac{d+2}{2}}_T L^{2(d+2)}_x} \lesssim T^{\,\frac{4-d}{2(d+2)}+} \, \langle N_3 \rangle^{\frac{d(d-1)}{2(d+2)}+}\,\|\phi_{3,N_3}\|_{L^2_x}\,.
\end{equation}
Estimating the $u_{1,N_1},u_{2,N_2}$ factors in \eqref{Case 2A Holder's inequality} as in Case 1, using \eqref{3rd_factor_Case2} as well as \eqref{choice_of_q} we obtain that the right-hand side of \eqref{Case2A product bound 2} is
\begin{equation}
\lesssim T^{\,\epsilon} \, \Big(\langle N_1 \rangle^{-\eta} \, \|\phi_{1,N_1}\|_{L^2_x} \Big) \, \Big(\langle N_2 \rangle^{\zeta_0+} \, \|\phi_{2,N_2}\|_{L^2_x} \Big) \, \Big(\langle N_3 \rangle^{\zeta_0+} \, \|\phi_{3,N_3}\|_{L^2_x}\Big)\,.
\end{equation}
We now conclude the argument as in Case 1.

\medskip

\textbf{(2)} $d \geq 5$.
In this case we use
\begin{align}
\notag
&\|u_{1,N_1} \, u_{2,N_2} \, u_{3,N_3}\|_{L^1_T L^{\frac{d}{d-1}}_x} 
\\
\label{Case2_d>4}
&\leq \|u_{1,N_1}\|_{L^{\frac{2(d+2)}{d}}_{T,x}} \, \|u_{2,N_2}\|_{L^{\frac{2(d+2)}{d}}_{T,x}} \, \|u_{3,N_3}\|_{L^{\frac{d+2}{2}}_T L^{\frac{d(d+2)}{d-2}}_x}\,.
\end{align}
By Lemma \ref{Bernstein_inequality_torus}, H\"{o}lder's inequality, and Proposition \ref{Strichartz Estimate} we obtain that
\begin{equation*}
\|u_{3,N_3}\|_{L^{\frac{d+2}{2}}_T L^{\frac{d(d+2)}{d-2}}_x} \lesssim T^{0+} \,\langle N_3 \rangle^{\zeta_0+} \, \|\phi_{3,N_3}\|_{L^2_x}
\end{equation*}
and we conclude the claim as in the case $d \leq 4$.

\medskip

\textbf{Case 2.B:} $N_1 \sim N_3 \gtrsim N_2$. This is analogous to Case 2.A  (interchange the roles $N_2$ and $N_3$).

\medskip

\textbf{Case 2.C:} $N_2 \sim N_3 \gtrsim N_1$.
In this case, we replace the power of $N_1$ in \eqref{Case2A product bound1B} by the corresponding power of $N_2$ and obtain
\begin{align*}
 \notag & \langle N \rangle^{-\eta} \, \|P_N(u_{1,N_1} \, u_{2,N_2} \, u_{3,N_3})\|_{L^1_T L^2_x}\\
\notag \lesssim{}& \langle N \rangle^{-\delta} \, \langle N_2 \rangle^{-\eta+\zeta_0+\delta} \, \|u_{1,N_1} \, u_{2,N_2} \, u_{3,N_3}\|_{L^1_T L^q_x}\\
\lesssim{}& \langle N \rangle^{-\delta} \, \langle N_1 \rangle^{-\eta} \, \langle N_2 \rangle^{\zeta_0+\delta} \, \|u_{1,N_1} \, u_{2,N_2} \, u_{3,N_3}\|_{L^1_T L^q_x}
\end{align*}
We can now argue as in Case 2.A. It is important to note that the factor of $N_1^{\,0+}$ coming from the application of Proposition \ref{Strichartz Estimate} can be absorbed into the powers of $N_2$ and $N_3$ in this case. This finishes the proof of \eqref{Star1} for the given choice of parameters.

\bigskip

As mentioned in Remark \ref{rmk:star1impliesstar2}, on $\D=\T^d_g$ the estimate \eqref{Star2} is a immediate consequence of \eqref{Star1}.
Indeed, if $S_3$ denotes the symmetric group, we have
\[
\|u_{1,N_1} u_{2,N_2} u_{3,N_3}\|_{L^1_T B^{\zeta}_{2,1}}\lesssim \sum_{\sigma \in S_3}\|\langle N_{\sigma(1)}\rangle^{\zeta} u_{\sigma(1),N_{\sigma(1)}} u_{\sigma(2),N_{\sigma(2)}} u_{\sigma(3),N_{\sigma(3)}}\|_{L^1_T B^{0}_{2,1}}.
\]
Now, \eqref{Star1} with $\eta=0$ yields
\[
\|u_{1,N_1} u_{2,N_2} u_{3,N_3}\|_{L^1_T B^{\zeta}_{2,1}}\lesssim \| \phi_{1,N_1}\|_{B^{\zeta}_{2,1}} \|\phi_{2,N_2}\|_{B^{\zeta}_{2,1}}\| \phi_{3,N_3}\|_{B^{\zeta}_{2,1}},
\]
and by summing up with respect to $N_1,N_2,N_3$ we obtain \eqref{Star2}.

\bigskip

We now turn to the proof of \eqref{Star3}. We start by noting that 
\begin{equation}
\label{Besov_embedding}
\|f\|_{B^{-\zeta_0}_{2,1}} \lesssim \|f\|_{L^q_x}\,,
\end{equation}
whenever $q \in (q_0,\infty]$ where $q_0$ is given by \eqref{choice_of_q}. By duality and compactness of the domain, \eqref{Besov_embedding} follows if we prove
\begin{equation}
\label{Besov Sobolev embedding bound duality}
\|f\|_{L^{q_0'-}_x} \lesssim \|f\|_{B^{\zeta_0}_{2,\infty}}\,
\end{equation}
Let $p:=q_0'-$. We then have $p>2$. In order to obtain \eqref{Besov Sobolev embedding bound duality} we use $P_N=\widetilde{P}_N P_N$ and Lemma \ref{Bernstein_inequality_torus} to deduce that
\begin{equation*}
\|f\|_{L^p_x}  
\lesssim_{\delta} \sup_{N}\, \langle N \rangle^{d/2-d/p+\delta} \|P_Nf\|_{L^2_x}= \|f\|_{B^{d/2-d/p+\delta}_{2,\infty}} \leq \|f\|_{B^{\zeta_0}_{2,\infty}}\,,
\end{equation*}
for $\delta$ sufficiently small. This proves \eqref{Besov Sobolev embedding bound duality} and hence \eqref{Besov_embedding}. To conclude the claim, we consider two cases, depending on the dimension.

\medskip

\textbf{(1)} $2 \leq d \leq 4$.
In this case, we use \eqref{Besov_embedding} and H\"{o}lder's inequality to obtain
\begin{equation*}
\|f_1 f_2 f_3\|_{B^{-\zeta_0}_{2,1}} 
\lesssim \|f_1\|_{L^{\frac{6(d+2)}{2d+1}+}_x} \,\|f_2\|_{L^{\frac{6(d+2)}{2d+1}+}_x}\,\|f_3\|_{L^{\frac{6(d+2)}{2d+1}+}_x}
\end{equation*}
and we deduce the claim from the Sobolev embedding
\begin{equation*}
H^{\frac{d \,(d+5)}{6(d+2)}+}_x \hookrightarrow L^{\frac{6(d+2)}{2d+1}+}_x\,,
\end{equation*}
see Proposition \ref{Sobolev_embedding} below.

\medskip

\textbf{(1)} $d \geq 5$.
In this case, we use
\begin{equation*}
\|f_1 f_2 f_3\|_{B^{-\zeta_0}_{2,1}} \lesssim \|f_1\|_{L^{\frac{3d}{d-1}+}_x} \, \|f_2\|_{L^{\frac{3d}{d-1}+}_x}  \,  \|f_3\|_{L^{\frac{3d}{d-1}+}_x}  
\end{equation*}
and the Sobolev embedding
\begin{equation*}
H^{\frac{d}{6}+\frac{1}{3}+}_x \hookrightarrow L^{\frac{3d}{d-1}+}_x
\end{equation*}
implies the claimed estimate.
\end{proof}

\begin{remark}\label{rmk:zhou}
One can also adapt the methods developed in \cite{Zhou} to prove an unconditional uniqueness result for \eqref{NLS}, but in this way, one obtains a worse range for $s$ than in \eqref{unconditional_uniqueness_torus_s}. In particular, this method does not give unconditional uniqueness in the energy class when $d=3$.
\end{remark}

\section{Uniqueness results in the 1d case and the Proof of Theorem \ref{thm:uu-1d}}
\label{sec:1d}
Let 
$I=[-T,T]$ for $T>0$, $\D=\T$ or $\D=\R$. Consider \eqref{NLS} in $1d$ 
with initial condition $\phi(0)=\phi_0 \in L^2(\D)$. 
It is well-posed, with uniqueness in some auxiliary space such as $X^{0,\frac{3}{8}}$ \cite{B93} or merely $L^4(I\times \D)$. The latter is  obvious in the case $\D=\R$, while in the case $\D=\T$ this follows from \cite{MV} and the Christ-Kiselev Lemma. Further, the problem is ill-posed below $L^2(\T)$, see \cite{Mol,OhSul}.

If $\D=\R$, it is proved in \cite{G2} that \eqref{NLS} is well-posed for initial data in Fourier-Lebesgue spaces. To be precise, let $\mu$ denote the Lebesgue measure on $\D^\ast=\R$ if $\D=\R$ and the counting measure on $\D^\ast=\Z$ if $\D=\T$. Suppose that $\psi(0)=\psi_0$ is a ($2\pi$-periodic if $\D=\T$) tempered distribution satisfying
\[
\|\phi_0\|_{\widehat{L^r}}:=\Big(\int |\widehat{\phi_0}(\xi)|^{r'}d\mu(\xi)\Big)^{\frac{1}{r'}}<+\infty,
\]
for $1<r\leq 2$, and $\frac{1}{r}+\frac{1}{r'}=1$. Following \cite{G,G2} we define the space $X^{s,b}_r$ via
\[
\|\phi\|_{X^{s,b}_r}=\Big(\int \ang{\xi}^{sr'}\int_\R \ang{\tau+|\xi|^2}^{br'}|\widehat{\phi}(\tau,\xi)|^{r'}d\tau d\mu(\xi)\Big)^{\frac{1}{r'}},
\]
and $X^{s,b}_r(I)$ is the restriction space. We refer the reader to \cite[Section 2]{G} for more details on these spaces. Now, if $\D=\R$, due to \cite[Theorem 1]{G2} there exists a unique solution $\psi\in X^{0,b}_r(I)$ of \eqref{NLS}, provided that $T>0$ is small enough, $1<r\leq 2$ and $b>\frac{1}{r}$.

In the case $\D=\T$, the key nonlinear estimate for the corresponding result in the case $\D=\T$ is proved in the introduction of \cite{GH}. This case is slightly more delicate and one needs to renormalize the equation first.
The well-known gauge transform \[\psi(t,x)=\exp\Big(i t \frac{1}{\pi} \int_0^{2\pi}|\phi(t,y)|^2dy \Big)\phi(t,x)\]
leads to the renormalized NLS
\begin{equation}\label{eq:1d-ren-nls}
i\partial_t\psi +\partial_x^2 \psi=\Big(|\psi|^2-\frac{1}{\pi}\int_0^{2\pi} |\psi(\cdot,y)|^2dy \Big) \psi \text{ on }(0,T)\times \T.
\end{equation}
The discussion in the introduction of
\cite{GH} implies that there is a trilinear form $N(\psi_1,\psi_2,\psi_3)$ with the property that
\[N(\psi,\psi,\overline{\psi})=\Big(|\psi|^2-\frac{1}{\pi}\int_0^{2\pi} |\psi(\cdot,y)|^2dy \Big)\psi .\]
\cite[Prop. 1.3]{GH} further implies that if $1<r\leq 2$, $b>\frac{1}{r}$, $\varepsilon>0$, then for all $\psi_j \in X^{0,b}_r$  we have $N(\psi_1,\psi_2,\overline{\psi_3})\in X^{0,-\varepsilon}_r$ and
\[
\|N(\psi_1,\psi_2,\overline{\psi_3})\|_{X^{0,-\varepsilon}_r}\lesssim \prod_{j=1}^3\|\psi_j\|_{X^{0,b}_r}.
\]
Then, by using the linear estimates in Lemma \ref{lem:lin} and the contraction mapping principle, there exists a unique solution $\psi\in X^{0,b}_r(I)$, provided that $T>0$ is small enough and $b>\frac{1}{r}$, see \cite{G2} or \cite[Section 2]{G} for details.

\begin{proof}[Proof of Theorem \ref{thm:uu-1d}]
By invariance under time translations we may assume that $|I|$ is sufficiently small.
By assumption $\phi$ satisfies
\begin{equation}\label{eq:mild-phi}
e^{-it \Delta}\phi(t)=\phi_0-i \int_0^t e^{-it'\Delta}|\phi|^2\phi(t')dt', \; t \in I,
\end{equation}
in some Sobolev space of negative order
and $\phi$ conserves the $L^2$-norm.
Let 
\[\psi(t,x)=e^{i t m}\phi(t,x), \qquad m=\frac{1}{\pi} \int_0^{2\pi}|\phi(t,y)|^2dy.\]
Then,
\begin{align*}
i\int_0^t e^{-it'\Delta} (|\psi|^2\psi(t')-m \psi(t'))dt'=&{}i\int_0^t e^{-it'\Delta} |\phi|^2\phi(t') e^{it' m}dt'\\&{}-\int_0^t e^{-it'\Delta} \phi(t')  \partial_{t'} e^{it' m}dt'
\end{align*}
Now, \eqref{eq:mild-phi} and integration by parts yield
\[-\int_0^t e^{-it'\Delta} \phi(t')  \partial_{t'} e^{it' m}dt'=\phi_0 - e^{-it\Delta} \phi(t) e^{it m}-i\int_0^t e^{-it'\Delta}|\phi|^2\phi(t') e^{it' m} dt',\]
which implies
\begin{align*}
i\int_0^t e^{-it'\Delta} (|\psi|^2\psi(t')-m \psi(t'))dt'=&\psi_0- e^{-it\Delta} \psi(t),
\end{align*}
which is equivalent to
\begin{equation}\label{eq:mild-psi}
\psi(t)= e^{it\Delta}\psi_0-i\int_0^t e^{i(t-t')\Delta} (|\psi|^2\psi(t')-m \psi(t'))dt', \; t\in I,
\end{equation}
hence $\psi$ is a mild solution of \eqref{eq:1d-ren-nls} in $L^2(\T)$ with
$\psi \in L^p(I\times \T)$.

Let $f=|\psi|^2\psi-m \psi$, extended by zero outside of $I$. It satisfies $f\in L^{r}(\R \times \T)$ with $r=p/3>1$. The statement of the Theorem is relevant only if $p>3$ is small, so without loss we may assume $p\leq 6$. By Hausdorff-Young we have $\widehat{f}\in  L^{r'}(\R\times \Z)$, in other words $f\in X^{0,0}_r$. Let $\chi \in C^\infty(\R)$ with $\chi(t)=1$ for $t\in [-1,1]$, $\supp(\chi)\subset (-2,2)$ and $\chi_T(t)=\chi(t/T)$ and let \[F(t) := \chi(t)e^{it\Delta}\psi_0-i\chi_T(t)\int_0^t e^{i(t-t')\Delta} f(t') dt'. \]
For $t\in I$ we have $F(t)=\psi(t)$. Observing $L^2(\T)\subset \widehat{L^{r}}(\T)$ for $r\leq 2$,
Lemma \ref{lem:lin} implies $F\in X^{0,1}_{r}$ and we obtain $\psi\in X^{0,1}_{r}(I)$. Due to the uniqueness result in $X^{0,1}_{r}(I)$ explained above the proof is complete.
\end{proof}

\appendix

\section{Proof of Proposition \ref{Product estimate proposition}}
\label{Appendix A}

For completeness, we present the details of the proof of Proposition \ref{Product estimate proposition}. The results here are well-known, some of them can be slightly improved on certain domains (such as tori). We first recall a general fact concerning the boundedness of the spectral projectors $\chi_k$ from the work of Sogge \cite{Sogge}.
Here $\mathbb{D}$ is a general smooth compact Riemannian manifold of dimension $d$ without boundary, which is connected and orientable.
\begin{proposition}
\label{Sogge_estimate}
Let $k \in \mathbb{N}$ be given. 
We have
\begin{equation}
\|\chi_k f \|_{L^{\infty}} \lesssim k^{(d-1)/2}\|f\|_{L^2}\,.
\end{equation}
\end{proposition}
This result is shown in \cite[Proposition 2.1]{Sogge}. 
Furthermore, we note that the non-endpoint Sobolev embedding holds on $\mathbb{D}$. This estimate suffices for our purposes.

\begin{proposition}[Non-endpoint Sobolev embedding]
\label{Sobolev_embedding}
The following results hold:
\begin{itemize}
\item[(a)]
Let $p \geq 2$ be given. Then we have
\begin{equation*}
\|f\|_{L^p} \lesssim_{p,s} \|f\|_{H^s}
\end{equation*}
for all $s>\frac{d}{2}-\frac{d}{p}$.
\item[(b)] Let $1 \leq p \leq 2$ be given. Then we have 
\begin{equation*}
\|f\|_{H^{-s}} \lesssim_{q,s} \|f\|_{L^p}
\end{equation*}
for all $s>\frac{d}{p}-\frac{d}{2}$.
\end{itemize}
\end{proposition}
Before proceeding to the proof of Proposition \ref{Sobolev_embedding}, we recall the Bernstein inequality on general domains $\mathbb{D}$.
\begin{lemma}[Bernstein's inequality on $\mathbb{D}$]
\label{Bernstein1_lemma}
For all $N$ and $2 \leq p \leq \infty$, we have
\begin{equation}
\label{Bernstein2_lemma}
\|P_N f\|_{L^p} \lesssim \langle N \rangle^{d/2-d/p} \|P_Nf\|_{L^2}\,.
\end{equation}

\end{lemma}
\begin{proof}[Proof of Lemma \ref{Bernstein1_lemma}]
We first prove \eqref{Bernstein2_lemma} for $p=\infty$. Note that, for fixed $N$
\begin{equation*}
\|P_N f\|_{L^{\infty}} \leq \sum_{k \sim N} \|\chi_kf\|_{L^\infty}\,,
\end{equation*}
which by Proposition \ref{Sogge_estimate} is
\begin{equation*}
\lesssim \sum_{k \sim N} k^{(d-1)/2} \|\chi_kf\|_{L^2}
\lesssim \langle N \rangle^{d/2} \, \|P_Nf\|_{L^2}\,.
\end{equation*}
In the second step above, we used the Cauchy-Schwarz inequality in $k$.
The claim for general $2 \leq p \leq \infty$ follows from \eqref{Bernstein2_lemma} with $p=\infty$ by using interpolation.
\end{proof}

\begin{proof}[Proof of Proposition \ref{Sobolev_embedding}]
We observe that part (b) follows from part (a) by duality. Therefore, it suffices to prove part (a). By Lemma \ref{Bernstein1_lemma} we have
\begin{equation*}
\|f\|_{L^p} 
\lesssim \sum_N \langle N \rangle^{d/2-d/p-s} \, \langle N\rangle^s\,\|P_Nf\|_{L^2} \lesssim \|f\|_{H^s}\,,
\end{equation*}
whenever $s>\frac{d}{2}-\frac{d}{p}$. 
\end{proof}

Before proving Proposition \ref{Product estimate proposition}, we prove an auxiliary estimate for the product of two functions.
\begin{lemma}
\label{Product estimate lemma 1}
Suppose that $\varrho_1,\varrho_2 \in (0,\frac{d}{2})$ are given. Then, for all $\delta>0$ we have
\begin{equation*}
\|f_1 f_2\|_{H^{\varrho_1+\varrho_2-\frac{d}{2}}} \lesssim_{\delta} \|f_1\|_{H^{\varrho_1+\delta}} \|f_2\|_{H^{\varrho_2+\delta}}\,.
\end{equation*}
\end{lemma}

\begin{proof}[Proof of Lemma \ref{Product estimate lemma 1}]
Let us write $\varrho:=\varrho_1+\varrho_2-\frac{d}{2}$. We consider two cases.

\medskip

\textbf{Case 1:} $\varrho>0$.
We note that by \eqref{Inclusion}, it suffices to estimate $\|f_1f_2\|_{B^{\varrho}_{2,1}}$, which by \eqref{Besov space M} and the triangle inequality is 
\begin{multline}
\leq \mathop{\sum_{N,N_1,N_2}}_{\mathrm{max}\{N_1,N_2\} \gtrsim N} \langle  N \rangle^{\varrho} \|P_N(P_{N_1}f_1P_{N_2}f_2)\|_{L^2} \,\,+  \mathop{\sum_{N,N_1,N_2}}_{N \gg N_1,N_2} \langle N \rangle^{\varrho} \|P_N(P_{N_1}f_1P_{N_2}f_2)\|_{L^2}
\\
=:I+II\,.
\label{I+II}
\end{multline}
We first estimate the term $I$. By symmetry it suffices to estimate the expression
\begin{equation*}
I':=\mathop{\sum_{N,N_1,N_2}}_{N_1 \gtrsim N} \langle N \rangle^{\varrho} \|P_N(P_{N_1}f_1P_{N_2}f_2)\|_{L^2}\,.
\end{equation*}
Since $\varrho>0$, we can sum in $N$ and obtain that
\begin{equation}
\label{I' bound}
I' \lesssim_{\varrho} \sum_{N_1,N_2} \langle N_1 \rangle^{\varrho} \|P_{N_1}f_1P_{N_2}f_2\|_{L^2}\,. 
\end{equation}
We choose $p_1,p_2$ such that $\frac{1}{p_1}=\frac{\varrho_2}{d}$ and $\frac{1}{p_2}=\frac{1}{2}-\frac{\varrho_2}{d}$. By assumption on $\varrho_2$, it follows that $p_1,p_2 \in (2,\infty)$. Hence, applying H\"{o}lder's inequality, Lemma \ref{Bernstein1_lemma}, and the identities
$\varrho+\tfrac{d}{2}-\frac{d}{p_1}=\varrho_1, \frac{d}{2}-\tfrac{d}{p_2}=\varrho_2$ in \eqref{I' bound} it follows that 
\begin{equation*}
I' \lesssim \sum_{N_1,N_2} \langle N_1 \rangle^{\varrho_1}\|P_{N_1}f_1\|_{L^2} \,\langle N_2 \rangle^{\varrho_2} \|P_{N_2}f_2\|_{L^2} = \|f_1\|_{B^{\varrho_1}_{2,1}} \|f_2\|_{B^{\varrho_2}_{2,1}}\,,
\end{equation*} 
which by \eqref{Inclusion} is an acceptable bound.

In order to estimate $II$ in \eqref{I+II} we first consider fixed $N$ and $N_1,N_2 \ll N$. Let us assume that $N_1 \leq N_2$.
We estimate the quantity $\|P_N(P_{N_1}f_1P_{N_2}f_2)\|_{L^2}$ by duality.
Namely, for $g \in L^2$ and $k \in \mathbb{N}$ we note that, by the self-adjointness of $\Delta$,
\begin{align}
\notag
&\bigg| \int P_{N_1} f_1 \, P_{N_2} f_2 \,P_N g \, \dd x \bigg| = \bigg| \int \Delta^k (P_{N_1} f_1 \, P_{N_2} f_2) \, \Delta^{-k} P_N g \, \dd x \bigg|
\\
\notag
&\leq \|\Delta^k (P_{N_1} f_1 \, P_{N_2} f_2)\|_{L^2} \, \|\Delta^{-k} P_N g\|_{L^2} 
\\
\label{Integration_by_parts_bound}
&\lesssim_k \langle N \rangle^{-2k}\,\langle N_1 \rangle ^{d/2} \, \langle N_2 \rangle^{2k} 
\|P_{N_1 f_1}\|_{L^2}\, \|P_{N_2} f_2\|_{L^2}\, \|g\|_{L^2}\,.
\end{align}
Note that, in order to obtain \eqref{Integration_by_parts_bound}, we used a Leibniz rule for $\Delta^k$, H\"{o}lder's inequality, Lemma \ref{Bernstein1_lemma} with $p=\infty$, together with the assumption that $N_1 \leq N_2$. 
Moreover, we used the observation that 
\begin{equation*}
\|\Delta^{-k} P_N\|_{L^2 \rightarrow L^2} \lesssim_k \langle N \rangle^{-2k}\,. 
\end{equation*}
We hence obtain from \eqref{Integration_by_parts_bound} and duality that for all $k \in \mathbb{N}$ 
\begin{equation}
\label{Integration_by_parts_bound2}
\|P_N(P_{N_1}f_1P_{N_2}f_2)\|_{L^2} \lesssim_k \langle N \rangle^{-2k}\,\langle N_1 \rangle ^{d/2} \, \langle N_2 \rangle^{2k} 
\|P_{N_1 f_1}\|_{L^2}\, \|P_{N_2} f_2\|_{L^2} \,.
\end{equation}
Let us observe that if $k>\tfrac{\varrho_2}{2}$, we have for $\delta>0$ 
\begin{align*}
&\mathop{\sum_{N,N_1,N_2}}_{N_1 \leq N_2 \ll N}\, \langle N \rangle^{\varrho_1+\varrho_2-d/2-2k} \, \langle N_1 \rangle^{d/2-\varrho_1-\delta} \, \langle N_2 \rangle^{2k-\varrho_2-\delta} \\
\lesssim{}& \mathop{\sum_{N,N_1,N_2}}_{N_1 \leq N_2 \ll N}\, \bigg(\frac{\langle N_1 \rangle}{\langle N \rangle}\bigg)^{d/2-\varrho_1} \, \langle N_1 \rangle^{-\delta} \, \langle N_2 \rangle^{-\delta}
\lesssim{} 1.
\end{align*}
Substituting this bound into \eqref{Integration_by_parts_bound2} and using \eqref{Inclusion}, we obtain the desired estimate on $II$ in \eqref{I+II}.

\medskip

\textbf{Case 2:} $\varrho \leq 0$. 
We let $q_1,q_2$ be such that $\frac{1}{q_1}=\frac{1}{2}-\frac{\varrho_1}{d}$ and $\frac{1}{q_2}=\frac{1}{2}-\frac{\varrho_2}{d}$. By the assumptions on $\varrho_1,\varrho_2$, it follows that $q_1,q_2 \in (2,\infty)$. Hence $\frac{1}{q_1}+\frac{1}{q_2}=\frac{1}{q}$ for some $q \in (1,\infty)$. For such $q$ we have
$\frac{1}{2}=\frac{1}{q}+\frac{\varrho}{d}$. Since $\varrho \leq 0$ we can use Proposition \ref{Sobolev_embedding} (b) with $s=-\varrho$ and $\tilde{q}=q+\delta_0$, with $\delta_0>0$, to deduce that
\begin{equation}
\label{Case2_bound}
\|f_1 f_2\|_{H^{\varrho}} \lesssim \|f_1f_2\|_{L^{\tilde{q}}}\,.
\end{equation}
Using H\"{o}lder's inequality, it follows that the right-hand side of \eqref{Case2_bound} is 
\begin{equation*}
\leq \|f_1\|_{L^{q_1+\delta_1}} \|f_2\|_{L^{q_2+\delta_2}}\,,
\end{equation*}
for appropriate $\delta_1,\delta_2>0$.
By Proposition \ref{Sobolev_embedding} (a), this expression is
\begin{equation*}
 \lesssim \|f_1\|_{H^{\varrho_1+\delta}} \|f_2\|_{H^{\varrho_2+\delta}}\,,
\end{equation*}
provided that we choose the parameters $\delta_0,\delta_1,\delta_2$ small in terms of $\delta$.
This is an admissible upper bound.
\end{proof}

We can now prove Proposition \ref{Product estimate proposition}.
\begin{proof}[Proof of Proposition \ref{Product estimate proposition}]
We note that, by the assumptions of the proposition, we have $2\varrho-\frac{d}{2} \in (0,\frac{d}{2})$. By applying Lemma \ref{Product estimate lemma 1} to $f_1f_2$ and $f_3$, we obtain
\begin{equation*}
\|f_1 f_2 f_3 \|_{H^{(2\varrho-\frac{d}{2})+\varrho-\frac{d}{2}}} \lesssim_{\delta} \|f_1f_2\|_{H^{(2\varrho-\frac{d}{2})+\delta/2}} \|f_3\|_{H^{\varrho+\delta/2}}\,,
\end{equation*}
We can hence apply Lemma \ref{Product estimate lemma 1} again to estimate the first factor above and obtain the claim.
\end{proof}

\section{Proof of Lemma \ref{Bernstein_inequality_torus}}\label{appendix-bern}
After an anisotropic rescaling of the torus we may assume that $\theta_1=\cdots=\theta_d=1$ and 
$\tilde{P}_N$ is of the form
\[
\tilde{P}_N f=\sum_{k \in \Z^d}\rho_N(k) \widehat{f}(k) e^{i x\cdot k}
\]
for $\rho_N(k)=\rho(k/N)$ if $N\geq 1$,  with some functions $\rho,\rho_0 \in \mathcal{S}(\R^d)$, and the standard $k$-th Fourier coefficient $\widehat{f}(k)$ of $f$.
By the Poisson summation formula we have
\[
\tilde{P}_N f =\psi_N \ast_{\T^d} f, \text{ where } \psi_N(x)= (2\pi)^{d/2}\,\sum_{k \in \Z^d} \mathcal{F}^{-1}(\rho_N)(x+k),
\]
where $\mathcal{F}$ denotes the Fourier transform on $\R^d$. It is enough to consider $N\geq 1$. Then,
\begin{align*}
\|\psi_N\|_{L^1([0,2\pi]^d)}\lesssim{}& \sum_{k \in \Z^d} \|\mathcal{F}^{-1}(\rho_N)\|_{L^1(k+[0,2\pi]^d)}\leq \|\mathcal{F}^{-1}(\rho_N)\|_{L^1(\R^d)}\\
={}& \|\mathcal{F}^{-1}(\rho)\|_{L^1(\R^d)}.
\end{align*}
Also, for any $x$ we have
\[|\psi_N(x)| \lesssim N^d \sum_{k\in \Z^d} \langle N (x+k)\rangle^{-d-1}\lesssim N^d,\]
and therefore
\[
\|\psi_N\|_{L^r([0,2\pi]^d)}^r \lesssim{}\|\psi_N\|_{L^1([0,2\pi]^d)}\|\psi_N\|^{r-1}_{L^\infty ([0,2\pi]^d)} \lesssim{}N^{d(r-1)}
\]
for any $1\leq r \leq \infty$.
If $1\leq p \leq q\leq \infty$, Young's inequality  implies
\[\|\tilde{P}_N f\|_{L^q} \lesssim \langle N\rangle^{d(1-1/r)} \|f\|_{L^p} \]
where $1-\frac{1}{r}=\frac{1}{p}-\frac{1}{q}$.
\section{Linear estimates in Fourier-Lebesgue spaces}\label{sec:lin-est}
Here, we include the linear estimates used in the proof of Theorem \ref{thm:uu-1d}, which are also needed to perform the contraction argument. These are due to Gr\"unrock, see \cite[Section 2]{G}. Here, we adapt the argument given in \cite[Proof of Lemma 7.1]{GH}, which follows the strategy of \cite[Lemma 3.1]{CKSTT}.
\begin{lemma}\label{lem:lin}
Let $1<r<\infty$, $b \geq 0$.  Let $\chi \in C^\infty(\R)$ with $\chi(t)=1$ for $t\in [-1,1]$, $\supp(\chi)\subset (-2,2)$ and $\chi_T(t)=\chi(t/T)$.
For all $0<T\leq 1$ and $\phi_0 \in \widehat{L^r}$ we have $\chi_T e^{it\Delta}\phi_0 \in X^{0,b}_r$ and
\[
\big\|\chi_T e^{it\Delta}\psi_0 \big\|_{X_r^{0,b}}\lesssim T^{\frac{1}{r}-b} \|\psi_0\|_{\widehat{L^r}}.
\] 
Further, if $\beta\geq b-1$ and $\tfrac{1}{r}>\beta>-\tfrac{1}{r'}$, for all $f \in X^{0,\beta}_r$ we have
\[\chi_T\int_0^t e^{i(t-t')\Delta} f(t') dt'\in X^{0,b}_r\]
and
\[\Big\|\chi_T \int_0^t e^{i(t-t')\Delta} f(t') dt'\Big\|_{X^{0,b}_r}\lesssim T^{1+\beta-b} \|f\|_{X^{0,\beta}_r}.\]
\end{lemma}
\begin{proof}
 If $\mathcal{F}$ denotes the space-time Fourier transform, we obtain
\[
\mathcal{F}(\chi_T(t)e^{it\Delta}\psi_0 )(\tau,\xi)=T\widehat{\chi}(T(\tau+|\xi|^2))\widehat{\psi_0}(\xi),
\]
and, using $\widehat{\chi}\in \mathcal{S}(\R)$,
\begin{equation}\label{eq:free}
\begin{split}
\|\chi_T e^{it\Delta}\psi_0 \|_{X_r^{0,b}}={}&\Big(\int |\widehat{\psi_0}(\xi)|^{r'}\int_\R \ang{\tau+|\xi|^2}^{br'}| T\widehat{\chi}(T(\tau+|\xi|^2))|^{r'}d\tau d\mu (\xi)\Big)^{\frac{1}{r'}}\\
\lesssim{}& T^{\frac{1}{r}-b}  \|\psi_0\|_{\widehat{L^r}},
\end{split}
\end{equation}
for any $b\in \R$. For the integrable function $\sigma_T(t')=\tfrac12 \chi_{4T}(t')\mathrm{sign}(t')$ and $|t'|, |t|\leq 2T$ we have $\sigma_T(t')+\sigma_T(t-t')=1$ if $0\leq t'\leq t$ and zero otherwise. Hence, we can decompose
\[
\chi_T(t)\int_0^t e^{i(t-t')\Delta} f(t') dt'=\chi_T(t)e^{it\Delta}I_1+\chi_T(t)I_2(t)\]
where
\begin{align*}
I_1=\int_\R  \sigma_T(t') e^{-it'\Delta} f(t') dt', \qquad
 I_2(t)=\int_\R e^{i(t-t')\Delta} \sigma_T(t-t') f(t') dt'.
\end{align*}
An application of \eqref{eq:free} gives
\[
\|\chi_Te^{it\Delta}I_1\|_{X_r^{0,b}}\lesssim T^{\frac{1}{r}-b} \|I_1 \|_{\widehat{L^r}}.
\]
By Parseval's identity,
\[
\mathcal{F}_x \int_\R  \sigma_T(t')e^{-it'\Delta} f(t') dt'(\xi)=\int_\R \overline{\mathcal{F}_t \sigma_T (\tau+|\xi|^2)}\mathcal{F} f(\tau,\xi)d\tau,
\]
and due to
$|\mathcal{F}_t \sigma_T (\tau)|\lesssim T \ang{T \tau}^{-1}$ we obtain
\[\|\chi_Te^{it\Delta}I_1\|_{X_r^{0,b}}\lesssim T^{1+\frac{1}{r}-b} \Big(\int\Big(\int_\R \ang{T(\tau+|\xi|^2)}^{-1}|\mathcal{F} f(\tau,\xi)|d\tau \Big)^{r'}d\mu (\xi)\Big)^{\frac{1}{r'}}.
\]
H\"older's inequality implies
\begin{align*}
\int_\R \ang{T(\tau+|\xi|^2)}^{-1}|\mathcal{F} f(\tau,\xi)|d\tau\lesssim{} T^{\beta-\frac{1}{r}} \Big(\int_\R \ang{\tau+|\xi|^2}^{\beta r'}|\mathcal{F} f(\tau,\xi)|^{r'}d\tau\Big)^{\frac{1}{r'}}
\end{align*}
for any $\tfrac{1}{r}>\beta>-\tfrac{1}{r'}$, which implies
\[
\|\chi_Te^{it\Delta}I_1\|_{X_r^{0,b}}\lesssim  T^{1+\beta-b} \|f\|_{X^{0,\beta}_r}.
\]
Concerning the second contribution we first observe that
\[
\ang{\tau+|\xi|^2}^b\lesssim \ang{\tau'+|\xi|^2}^b+|\tau-\tau'|^b,
\]
which implies
\[\|\chi_T I_2 \|_{X_r^{0,b}}\lesssim \|\widehat{\chi_T}\|_{L^1}\|I_2\|_{X_r^{0,b}}+\||\cdot|^b\widehat{\chi_T}\|_{L^1}\|I_2\|_{X_r^{0,0}}.\]
We have $\||\cdot|^a\widehat{\chi_T}\|_{L^1}\lesssim T^{-a}$ for any $a\geq 0$, and 
\[
\mathcal{F}I_2(\tau,\xi)=\mathcal{F}_t \sigma_T(\tau+|\xi|^2)\mathcal{F}f(\tau,\xi).
\]
We obtain
\[
\|I_2\|_{X_r^{0,b}}\lesssim T \Big(\int \int_\R \frac{\ang{\tau+|\xi|^2}^{r'b}}{\ang{T(\tau+|\xi|^2)}^{r'}}|\mathcal{F} f(\tau,\xi)|^{r'}d\tau d\mu (\xi)\Big)^{\frac{1}{r'}}\lesssim T^{1+\beta-b}\|f\|_{X_r^{0,\beta}},
\]
which implies the claim.
\end{proof}

\subsection*{Acknowledgments} The authors would like to thank the
referee for valuable comments. S.H. thanks Tadahiro Oh for a helpful
discussion about \cite{GuoKwonOh}. Support by the German Science Foundation, IRTG 2235, is gratefully acknowledged.


\begin{thebibliography}{99}
\bibitem{Ammari_Liard_Rouffort} Z. Ammari, Q. Liard, C. Rouffort, \emph{On well-posedness and uniqueness for general hierarchy equations of Gross-Pitaevskii and Hartree type}, \texttt{https://arxiv.org/abs/1802.09041}.
\bibitem{B93} J. Bourgain, \emph{Fourier transform restriction phenomena for certain lattice subsets and applications to nonlinear evolution equations. I. Schr\"odinger equations},
Geom. Funct. Anal. 3 (1993), no. 2, 107--156.  
\bibitem{Bourgain_Demeter1} J. Bourgain, C. Demeter, \emph{The proof of the $l^2$ Decoupling Conjecture}, Ann. of Math. (2) 182 (2015), no. 1, 351--389.
\bibitem{Bulut_Czubak_Li_Pavlovic_Zhang} A. Bulut, M. Czubak, D. Li, N. Pavlovi\'{c}, X. Zhang, \emph{Stability and unconditional uniqueness of solutions for energy critical wave equations in high dimensions}, Comm. Partial Diff. Eqs. \textbf{38} (2013), no. 4, 575--607.
\bibitem{Burq_Gerard_Tzvetkov_2005} N. Burq, P. G\'{e}rard, N. Tzvetkov, \emph{Bilinear eigenfunction estimates and the nonlinear Schr\"{o}dinger equation on surfaces}, Invent. Math. \textbf{159} (2005), no. 1, 187--223.
\bibitem{Caz} T. Cazenave, Semilinear {S}chr\"odinger equations, Courant Lecture Notes in Mathematics 10, American Mathematical Society, Providence,
              RI, 2003. xiv+323 p.
\bibitem{Cazenave_Han_Fang} T. Cazenave, D. Fang, Z. Han, \emph{Local well-posedness for the $H^2$ critical nonlinear Schr\"{o}dinger equation}, Trans. Amer. Math. Soc. \textbf{368} (2016), no. 11, 7911--7934.
\bibitem{ChHaPavSei} T. Chen, C. Hainzl, N. Pavlovi\'{c}, R. Seiringer, \emph{Unconditional uniqueness for the cubic Gross-Pitaevskii hierarchy via quantum de Finetti}, Comm. Pure Appl. Math. \textbf{68} (2015), no. 10, 1845--1884.
\bibitem{CP1} T. Chen, N. Pavlovi\'{c}, \emph{On the Cauchy problem
    for focusing and defocusing Gross-Pitaevskii hiearchies},
  Discr. Contin. Dyn. Syst. \textbf{27} (2010), no. 2, 715--739.
\bibitem{CKSTT}
J. Colliander, M. Keel, G. Staffilani, H. Takaoka,   and T. Tao,
     \emph{Multilinear estimates for periodic {K}d{V} equations, and
              applications}, J. Funct. Anal. \textbf{211} (2004) no. 1, 173--218.
\bibitem{EESY} A. Elgart, L. Erd\H{o}s, B. Schlein, H.-T. Yau, \emph{Gross-Pitaevskii Equation as the Mean Field Limit of Weakly Coupled Bosons}, Arch. Rational Mech. Anal. \textbf{179} (2006), 265--283.
\bibitem{ESY2} L. Erd\H{o}s, B. Schlein, H.-T. Yau, \emph{Derivation
    of the cubic non-linear Schr\"{o}dinger equation from quantum
    dynamics of many-body systems}, Invent. Math. \textbf{167} (2007),
  no. 3, 515--614.
\bibitem{Farah} L.G. Farah, \emph{Local solutions in Sobolev spaces and unconditional well-posedness for the generalized Boussinesq equation}, Commun. Pure Appl. Anal. \textbf{8} (2009), no. 5, 1521--1539.
\bibitem{FurioliTerraneo} G. Furioli, E. Terraneo, \emph{Besov Spaces and Unconditional Well-Posedness for the Nonlinear Schr\"{o}dinger Equation in $\dot{H}^s(\mathbb{R}^n)$},  
Commun. Contemp. Math. \textbf{5} (2003), no. 3, 349--367. 
\bibitem{FurioliTerraneoPlanchon} G. Furioli, E. Terraneo, F. Planchon, \emph{Unconditional well-posedness for semilinear Schr\"{o}dinger equation and wave equations in $H^s$}, (English summary) Harmonic analysis at Mount Holyoke (South Hadley, MA, 2001), 147--156, Contemp. Math. \textbf{320}, Amer. Math. Soc., Providence, RI, 2003. 
\bibitem{G} A. Gr\"unrock, \emph{An improved local well-posedness result for the modified
              {K}d{V} equation}, Int. Math. Res. Not. \textbf{2004} (2004), no. 61, 3287--3308.
\bibitem{G2} A. Gr\"unrock, \emph{Bi- and trilinear {S}chr\"odinger estimates in one space
              dimension with applications to cubic {NLS} and {DNLS}}, Int. Math. Res. Not. \textbf{2005} (2005), no. 41, 2525--2558.
\bibitem{GH} A. Gr\"unrock, S. Herr,
\emph{Low regularity local well-posedness of the derivative nonlinear Schr\"odinger equation with periodic initial data},
SIAM J. Math. Anal. 39 (2008), no. 6, 1890--1920.
\bibitem{GuoKwonOh} Z. Guo, S. Kwon, T. Oh, \emph{Poincar\'{e}-Dulac normal form reduction for unconditional well-posedness of the periodic cubic NLS}, Comm. Math. Phys. \textbf{322} (2013), no. 1, 19--48. 
\bibitem{Herr_2013} S. Herr, \emph{The quintic nonlinear Schr\"{o}dinger equation on three-dimensional Zoll manifolds}, Amer. J. Math. \textbf{135} (2013), no. 5, 1271--1290.
\bibitem{HanFang} Z. Han, D. Fang, \emph{On the Unconditional Uniqueness for NLS in $\dot{H}^s$}, SIAM J. Math. Anal., \textbf{45} (2013), no. 3, 1505--1526.
\bibitem{HerrSohinger} S. Herr, V. Sohinger, \emph{The Gross-Pitaevskii hierarchy on general rectangular tori}, Arch. Rational Mech. and Anal., \textbf{220} (2016), no. 3, 1119--1158. 
\bibitem{HTT} S. Herr, D. Tataru, N. Tzvetkov, \emph{Global well-posedness of the energy critical nonlinear Schr\"{o}dinger equation with small initial data in $H^1(\mathbb{T}^3)$}, Duke Math. J. \textbf{159} (2011), no. 2, 329--349.
\bibitem{HTX} Y. Hong, K. Taliaferro, Z. Xie, \emph{Unconditional Uniqueness of the cubic Gross-Pitaevskii Hierarchy with Low Regularity}, SIAM J. Math. Anal. \textbf{47} (2015), no. 5, 3314--3341. 
\bibitem{HTX2} Y. Hong, K. Taliaferro, Z. Xie, \emph{Uniqueness of solutions to the 3D quintic Gross-Pitaevskii hierarchy}, J. Funct. Anal., \textbf{270} (2016), no. 1, 34--67.
\bibitem{K95} T. Kato, \emph{On nonlinear {S}chr\"odinger equations. {II}. {$H^s$}-solutions
              and unconditional well-posedness}, J. Anal. Math. \textbf{67} (1995), 281--306.
\bibitem{K96} T. Kato, \emph{Correction to: ``{O}n nonlinear {S}chr\"odinger equations. {II}.
              {$H^s$}-solutions and unconditional well-posedness''}, J. Anal. Math. \textbf{68} (1996), 305--305.
\bibitem{Killip_Visan} R. Killip, M. Vi\c{s}an, \emph{Scale invariant Strichartz estimates on tori and applications}, Math. Res. Lett. \textbf{23} (2016), no. 2, 445--472.
\bibitem{Kishimoto} N. Kishimoto, Private communication.
\bibitem{KM} S. Klainerman, M. Machedon, \emph{On the uniqueness of solutions to the Gross-Pitaevskii hierarchy}, Comm. Math. Phys. \textbf{279}, no.1 (2008), 169--185.
\bibitem{Lu_Xu} J. Lu, Y. Xu, \emph{Unconditional uniqueness of solution for $\dot{H}^{s_c}$ critical NLS in high dimensions}, J. Math. Anal. Appl., \textbf{436} (2016), no. 2, 1214--1222.
\bibitem{Masmoudi_Nakanishi} N. Masmoudi, K. Nakanishi, \emph{Uniqueness of solutions for Zakharov systems}, Funkcial. Ekvac. \textbf{52} (2009), no. 2, 233--253.
\bibitem{Masmoudi_Planchon} N. Masmoudi, F. Planchon, \emph{Unconditional well-posedness for wave maps}, J. Hyperbolic Differ. Equ. \textbf{9} (2012), no. 2, 223--237.
\bibitem{Mol} L. Molinet, \emph{On ill-posedness for the one-dimensional periodic cubic Schr\"odinger equation},
Math. Res. Lett. \textbf{16} (2009), no. 1, 111--120.
\bibitem{MV} A. Moyua and L. Vega, \emph{Bounds for the maximal function associated to periodic solutions of one-dimensional dispersive equations},
Bull. Lond. Math. Soc. 40 (2008), no. 1, 117--128.
\bibitem{OhSul} T. Oh, C. Sulem, \emph{On the one-dimensional cubic nonlinear Schr\"odinger equation below $L^2$},
Kyoto J. Math. 52 (2012), no. 1, 99--115. 
\bibitem{Planchon} F. Planchon, \emph{On uniqueness for semilinear wave equations}, Math. Z. \textbf{244} (2003), no. 3, 587--599. 
\bibitem{Rogers} K.M. Rogers, \emph{Unconditional well-posedness for subcritical NLS in $H^s$},   C. R. Math. Acad. Sci. Paris \textbf{345} (2007), no. 7, 395--398.
\bibitem{Sogge} C. D. Sogge, \emph{Concerning the $L^p$ Norm of Spectral Clusters for Second-Order Elliptic Operators on Compact Manifolds}, J. Funct. Anal. \textbf{77} (1988), 123--138.
\bibitem{VS2} V. Sohinger, \emph{A rigorous derivation of the defocusing cubic nonlinear Schr\"{o}dinger equation from the dynamics of many-body quantum systems}, Ann. Inst. H. Poincar\'{e} C, Analyse Non-Lin\'{e}aire, \textbf{32} (2015), no. 6, 1337--1365.
\bibitem{2S} H. Spohn, \emph{Kinetic equations from Hamiltonian Dynamics}, Rev. Mod. Phys. \textbf{52} (1980), no. 3, 569--615.
\bibitem{Stein_Weiss} E. Stein, G. Weiss, \emph{Introduction to Fourier Analysis on Euclidean Spaces}, Princeton University Press (1971).
\bibitem{Tao} T. Tao \emph{Global existence and uniqueness results for weak solutions of the focusing mass-critical nonlinear Schr\"{o}dinger equation}, Anal. PDE \textbf{2} (2009), no. 1, 61--81.
\bibitem{Win_Tsutsumi} Y.Y.S. Win, Y. Tsutsumi, \emph{Unconditional uniqueness of solution for the Cauchy problem of the nonlinear Schr\"{o}dinger equation}, Hokkaido Math. J. \textbf{37} (2008), no.4, 839--859.
\bibitem{Zhou} Y. Zhou, \emph{Uniqueness of weak solution of the KdV equation}, Int. Math. Res. Notices \textbf{6} (1997): 271--283.
\end{thebibliography}
\end{document}